\documentclass[12pt]{amsart}
\usepackage{graphicx} 
\usepackage[top=3cm, bottom=5cm, left=3cm, right=3cm]{geometry}

\usepackage[english]{babel}
\usepackage{amsmath}
\usepackage{amsthm}
\usepackage{amsfonts}
\usepackage{graphicx}
\usepackage{tikz-cd}
\usepackage[colorlinks=true, allcolors=black]{hyperref}
\usepackage{comment}
\usepackage{mathtools}
\usepackage[normalem]{ulem}
\newtheorem{theorem}{Theorem}[section]
\newtheorem{lemma}[theorem]{Lemma}
\newtheorem{prop}[theorem]{Proposition}

\newtheorem{cor}[theorem]{Corollary}
\newtheorem{conj}[theorem]{Conjecture}
\theoremstyle{definition}
\newtheorem{definition}[theorem]{Definition}
\newtheorem{remark}[theorem]{Remark}
\newtheorem{remarks}[theorem]{Remarks}

\newtheorem{examples}[theorem]{Examples}

\numberwithin{equation}{section}

\newcommand{\bZ}{{\mathbb Z}}
\newcommand{\bR}{\mathbb R}
\newcommand{\bC}{{\mathbb C}}

\newcommand{\ra}{\rightarrow}
\newcommand{\lr}{{\langle.,. \rangle}}

\title{Polarization complete left invariant  connections}
\author{Balázs Forman}
\address{ELTE - Eötvös L. Univ., Dept. of Anal.,
 P\'azm\'any P\'eter s\'et\'any 1/C, 1117, Budapest, Hungary}
\email{formanbalazsattila@gmail.com}

\author{Róbert Szőke}
\address{ELTE - Eötvös L. Univ., Dept. of Anal.,
 P\'azm\'any P\'eter s\'et\'any 1/C, 1117, Budapest, Hungary}
\email{rszoke2009@gmail.com}
\thanks{The second author was supported by ERC Advanced Grant KnotSurf4d}

\begin{document}
\raggedbottom
\begin{abstract}
Let $(M,\nabla)$ be a real analytic Koszul manifold. An 
 adapted complex structure (ac-structure)
on a neighborhood $N$ of the zero section in $TM$ is a complex structure on $N$ such that the leaves of the Levi-Civita foliation are holomorphic curves. More generally, a complex polarization $P$ on $N$  is called an ac-polarization if the leaves of the Levi-Civita foliation are tangential to $P$. The bundle of (1,0) tangent vectors of an ac-structure is an ac-polarization. 
The connection is called entire (resp. polarization complete or simply  $\mathcal P$-complete) if the ac-structure (resp. the ac-polarization) exists on $TM$. 
Although on a small enough $N$ 
an ac-structure always exists, entire connections are rear and if a maximal domain of definition $N_{max}$ of an ac-structure exists (different from $TM$), $N_{max}$ is a complicated  domain. On the other hand in many cases the associated  ac-polarization can be extended to the whole $TM$. The main purpose of the paper is to gain better understanding of this phenomenon using a generalized version of the polar map. As special cases we show that many of those  metrics studied by Aslam-Burns-Irvine and Halverscheid-Iannuzzi although  are not entire but are  $\mathcal P$-complete.
\end{abstract}

\maketitle
\section{Introduction}
Every real-analytic manifold $M$ admits a complexification, i.e. a complex manifold $Z$ into which   $M$ embeds as a   maximally totally real submanifold  (\cite{WB59}, \cite{Sh58}. $Z$ is far from being unique, complexifications exist with  quite different topology. Only the germ of $Z$ along $M$  is unique.
 If a real analytic metric $g$ is given on $M$, one gets a canonical complex structure, called {\it adapted complex structure} (or ac-structure for short) on some open neighborhood $N$ of the zero section in $TM$.
 The compatibility condition imposed on $g$ is that the leaves of the Riemann foliation of $TM\setminus M$ should be holomorphic curves (more precisely the intersection of each leaf with $N$). 
 Uniqueness of such complex structures was shown in \cite{LSz91} and existence in \cite{Sz91}.
 Important  properties of the ac-structure
 are (in fact these properties were the main motivations ): the energy function $E(v)=g(v,v)/2$ is strictly plurisubharmonic, $u=\sqrt{E}$ is plurisubharmonic and satisfies the homogeneous complex Monge-Amp\`ere (HCMA) equation on $N\setminus M$.  \cite{LSz91}. 
 A different but equivalent construction was given in \cite{GS91}  where $N$ was not in $TM$ but rather in $T^*M$. 
 
  The metric $g$ is said to be $\it entire$, if the ac-structure can be defined on $TM$. Compact manifolds with entire $g$ are rare, and notoriously difficult to construct. This is a very strong property of the metric.
  A necessary condition (\cite{LSz91} is that the sectional curvatures of $g$ must be nonnegative (for this to hold in fact completeness of $g$ suffices), but it is far from sufficient. A 2-parameter family of entire metrics was constructed on $S^2$ in \cite{Sz91}, but it is still not known if they are all. Other properties showing that entire metrics are very special: the geodesic flow has zero entropy ($M$ compact, orientable) \cite{SS24}, $M$ is rationally elliptic, hence  $\chi(M)\ge0$ ($M$ compact, simply connected) \cite{Ch24}.
   A related conjecture of Aguilar-Burns says that for $(M,g)$ compact entire, $(TM,J_{ac})$ should be affine algebraic.
  The first examples  of entire metrics were given by Patrizio and Wong in \cite{PW91}
  (although their terminology was different), where they showed that 
  every compact rank-1 symmetric space- metric is entire. This was generalized in \cite{Sz91}
  to compact symmetric spaces, and later in \cite{Sz98} to all compact, normal Riemannian homogeneous metrics.  A further generalization was given by
   Aguilar \cite{Ag01}, who proved  that the quotient of a compact $M$ with entire metric w.r.t. a compact Lie subgroup of the isometry group is again entire (assuming the quotient is smooth).
   Finally Lempert in \cite{L19b} showed that the submersion of an entire metric is again entire. Products of entire metrics are of course again entire.
 It is striking that most of the constructions that produces nonnegatively curved metrics also produces entire metrics. Therefore it is tempting to ask if the reverse is true. Namely if a compact manifold admits a metric with nonnegative curvature, does it follow that it also admits an entire metric? We are not aware of any result in this direction. There is a related conjecture of Totaro \cite{T03} on good complexifications.
   
  A necessary and sufficient condition for $g$ to be entire, was given in  \cite{Sz91} in terms of certain matrices obtained from Jacobi fields along an arbitrary geodesic (in the surface case it says: the quotient of two normal Jacobi field should extend meromorphically to $\bC$ with positive imaginary part in the upper half plane). This was  enough to conclude that compact symmetric spaces are entire, but the condition is very difficult to check in concrete examples.
  A different method was used in \cite{Sz95}, \cite{Sz98}, where the existence of the universal complexification $G_\bC$ of a compact Lie group was exploited and shown
  that the pull back of the complex structure of $G_\bC$ (resp. of $G_\bC/K_\bC$) by the  polar map $TG\ra G_\bC$ (resp. $T(G/K)\ra G_\bC/K_\bC$) (known to be a diffeomorphism in these cases)  is adapted to the biinvariant (resp. normal homogeneous) metric.

  Although in \cite{LSz91}, \cite{Sz91} 
 $M$ was compact and $N$  a disk subbundle in $TM$ with fixed  radius, it was clear at the beginning that
 uniqueness of ac-structures
could be proved in more general situations and that they extend (in certain cases) to more general domains.
What was not (and it is still not) clear, how to define some kind of maximal domain of definition where the ac-structure lives. A domain $N\subset TM$ is called a $*$-domain, if for all $v\in N$ and $0\le s<1$, $sv\in N$.
Intersection and union of $*$-domains are again $*$-domains and  uniqueness of ac-structure holds for $*$-domains. Hence there always exists a maximal $*$-domain $N_*\subset TM$ where the ac-structure exists, but in principle an ac-structure could live on an even larger domain.  
   It could even happen (although we have no explicit example of this kind) that analytic continuation of $J^{ac}:T_vN_*\ra T_vN_*$
along different curves starting at $v$ and ending
at $w\in TM\setminus N_*$ yields different complex structures in a neighborhood of $w$.
In any case, $N_*$ can have very complicated boundary and can behave  quite badly from a complex geometric point of view. It can be determined explicitly only for very special metrics.
For a Riemannian symmetric space of noncompact type $N_*$ is Stein and it is biholomorphic to the Akhiezer-Gindikin domain \cite{BHH03}. On the other hand
 for  a   left invariant Riemannian metric on the 3-dimensional Heisenberg group and on generalized Heisenberg groups $N_*$ is neither holomorphically separable nor holomorphically convex \cite{HI03}. 
 
 After some time it became clear that  positivity of the metric is not important,
  one could take pseudo-Riemannian metrics or simply any real analytic Koszul connections,  the notion of ac-structure naturally generalizes  \cite{Bi03}, \cite{Sz04}. We no longer have a Monge-Amp\`ere solution though, only the characteristic foliation survives.
  For $*$-domains we still have uniqueness and  existence (on a small enough domain). 
  In \cite{HI06} and \cite{HI09} the ac-structure of a 1-parameter family of left invariant pseudo-Riemannian metrics 
  on semisimple Lie groups  was studied. Their  metrics were  deformations of the Killing form. They gave a group theoretical description of $N_*$ in each case, again using a version of the polar map.
In the case of $SL(2,\bR)$ they demonstrated the highly complicated and parameter dependent complex geometric  behavior of $N_*$  in great details.  

 The $T^{1,0}N$ tangent bundle of an ac-structure that originally is defined on a maximal possible domain $N\subset TM$ in certain cases can be extended across the boundary of $N$ as an involutive bundle.  This phenomenon was observed in \cite{Sz01} for  Riemannian symmetric spaces of noncompact type, and later for affine locally symmetric Koszul connections  \cite{Sz04}. In fact in these cases the extended bundle will be defined over $TM$. We shall say that a Koszul connection is $\mathcal P$-complete if its ac-structure behaves this way. 

 One of the main purpose of this paper to demonstrate that the examples mentioned above in \cite{HI03}, \cite{HI06}, \cite{HI09},
 some of the metrics on $SU(2)$ investigated in \cite{ABI18} and many more are although not entire but are $\mathcal P$-complete. 
 In this paper we study  ac-structures with the help of a generalized version of polar maps, inspired by the paper \cite{L19a} of Lempert.  For a Koszul manifold $(M,\nabla)$, complex manifold $Z$ and smooth map $f:M\ra Z$, denote by $N_i(f,Z)\subset TM$ the set of those tangent vectors $v$, for which the $\nabla$-geodesic $\gamma_v$, with initial condition $\gamma'_v(0)=v$ extends holomorphically to some neighborhood of $[0,i]\subset\bC$. Denoting by $\gamma^\bC_v$ this extension, the polar map is defined as
 $\psi_f:N_i(f,Z)\ra Z$, $v\mapsto \gamma^\bC_v(i)$.
As it is explained in Subsection~\ref{Ss:motiv}, these are intimately related to an ac-structure (or ac-polarization) on $N_i(f,Z)$. We would like to choose $Z$ and $f$, so that $N_i(f,Z)$ is as large as possible. 

 The organization of the paper as follows.
 In Sect.2 we present the main definitions.
 In Sect.3 Theorem~\ref{T:pol} we show that a holomorphic foliation on a complex manifold induces a natural involutive bundle on any maximally totally real submanifold. As an application in Theorem~\ref{T:subm} we 
 prove that a holomorphic submersion $X\ra Z$
 and a maximallly totally real immersion $\varphi:N\ra X$ induces an involutive bundle on $N$. These results will be used to prove one of the main result,  Theorem~\ref{T:polarization}.
 In Sect.4 we collect some known and less known facts on flows of real or holomorphic vector fields, in particular on holomorphic geodesic flows of holomorphic Koszul connections. 
 
 In Sect.5 we introduce the notion of complexification of a real-analytic Koszul manifold $(M,\nabla)$, as a triple $(Z,\mathfrak D,\iota)$, where $(Z,\mathfrak D)$ is a holomorphic Koszul manifold, $\iota:M\ra Z$ a maximally totally real immersion with $\nabla=\iota^*{\widehat\nabla}$, where $\widehat\nabla$ is the induced real connection (by $\mathfrak D$) on $Z$. 
 Here our main result is Theorem~\ref{T:polarization}, which says that with the help of the polar map $\psi_\iota:N_i(Z,\iota)\ra Z$ of a complexification,  $P:=\psi_\iota^*(T^{1,0}Z)$ yields an ac-polarization 
 on $N_i(Z,\iota)$.
As a corollary we obtain Theorem~\ref{T:icompl}: if $(M,\nabla)$  admits an $i\bR$-complete complexification, then $\nabla$ is $\mathcal P$-complete. $i\bR$-completeness means that the curves $\iota\circ\gamma_v$ in $Z$ extends to a holomorphic map defined in some neighborhood of $i\bR$ for $\forall v\in TM$.

In Sect.6 we study left invariant Koszul connections both for real, as well as complex Lie groups. We extend Arnold's method of reducing the geodesic equation to the problem of finding integral curves of a vector field defined on the Lie algebra.

In Sect.7 we prove another important theorem (Theorem~\ref{T:eta-polarization}):  for every Lie group $G$ and left invariant Koszul connection $\nabla$, $N_i(\eta,G_\bC)$ is open in $TG$, where
$\eta:G\ra G_\bC$ is the universal complexification. Let  $\psi_\eta$ be the polar map. Then   $P:=(\psi_\eta)_*^{-1}(T^{1,0}G_\bC)$ defines an ac-polarization. In particular if the Euler-Arnold vector field is $i\bR$-complete, $\nabla$ will be $\mathcal P$-complete (Theorem~\ref{T:nabla-i-complete}). A corollary  of the latter theorem, is that the canonical connection of every Lie group is $\mathcal P$-complete, although not always entire. Another corollary is the fact, that a certain 1-parameter family of  left invariant (pseudo)-Riemannian metrics, in the Riemannian case  studied in \cite{ABI18}, are all $\mathcal P$-complete, (see Theorem~\ref{T:SU(2)} and Remarks~\ref{r:SU(2)} for  more precise statements).

 Sect.8 is a short reminder of biinvariant connections. In Sect.9 we give an explicit formula for all biinvariant connections  on an $\mathcal F$-type Lie group and show that they are all $\mathcal P$-complete. We also prove that on such groups all left invariant pseudo-Riemann metrics are $\mathcal P$-complete, but their Euler-Arnold vector field is not always $i\bR$-complete.
 Sect.10 is on 2-step nilpotent Lie groups. Here we give  a large family of $\mathcal P$-complete biinvariant connections. For certain 2-step Lie groups we can prove that these are all the biinvariant connections. We also prove that on a complex 2-step nilpotent Lie group every leftinvariant holomorphic Riemann metric is $\bC$-complete. As a corollary of this is:  on a real 2-step nilpotent Lie group all left invariant pseudo-Riemannian metric is $\mathcal P$-complete. As an application  we get that  the metrics studied by Halverscheid-Iannuzzi in \cite{HI03} on the 3-dimensional Heisenberg group and on generalized Heisenberg groups are not entire but all are $\mathcal P$-complete.
 In Sect. 11 we study Lie groups whose Lie algebra admits a reductive decomposition.
 We give a family of left invariant connections, a 1-parameter deformation of the canonical connection, whose polar maps can be given very explicitly and are all $\mathcal P$-complete (Theorem~\ref{T:reddef}). As a corollary we obtain that the 1-parameter metric deformation of the Killing form on semisimple Lie groups studied by Halverscheid-Iannuzzi in \cite{HI06}, \cite{HI09} are all $\mathcal P$-complete, although not all are entire (Theorem~\ref{T:Cheegerdeformmetriccompl} and Subsect.~\ref{Ss:semisimple}).
 In Sect.12 we collected certain calculations on vector fields that we used throughout the paper. 
 
\noindent \textbf{Acknowledgement.} Parts of this work was done while RSz
was visiting the Rényi Institute and Erdős Center. He would like to thank for these institutions for their financial support, hospitality and very inspiring scientific atmosphere. 
\section{Adapted complex and involutive structures}
\subsection{Notations}\label{Ss:notations}
For  $I=(a,b)$, $a<0<b$ let
\begin{equation}\label{E:ueps}
\bC I=\{z=t+is\in\mathbb C \mid \Re z=t\in I\}.
\end{equation}
Let $M$ be a smooth manifold.
We shall denote by  $N\subset TM$  an open, fiberwise $*$-domain (i.e. $v\in N$ and $0\le s\le1$ implies $sv\in N$)  with $M\subset N$.
Let $\nabla$ be a  Koszul connection  on $TM$. 
For $v\in TM$ denote by 
$\gamma_v:I\ra M$  the unique geodesic
with initial condition $\dot\gamma_v(0)=v$. Define
\begin{equation}\label{F:Omegav}
\begin{aligned}
\Omega_v:\bC I&\ra TM\\
t+is&\mapsto s\dot{\gamma}(t)
\end{aligned}
\end{equation}
The images $\Omega_v(\bC I\setminus I)$, $v\in TM\setminus M$ gives a foliation of $TM\setminus M$, called the {\it Levi-Civita foliation} (and when $\nabla$ is the Levi-Civita connection of the metric $g$, it is called the {\it Riemann foliation}).
\begin{definition} 
  An {\it adapted complex structure} (ac-structure for short) on $N$ means a  complex manifold structure so that for 
$\forall v\in TM$, the map $\Omega_v$ is holomorphic on $\Omega_v^{-1}(N)$.
$\nabla$ is called {\it entire}, if an ac-structure exists on $TM$.
A  pseudo-Riemannian manifold $(M,g)$    is called {\it entire} if its Levi-Civita connection  is entire.
\end{definition}
For compact $M$ existence and uniqueness  questions were treated in \cite{LSz91},  \cite{Sz91} when $\nabla$ is the Levi-Civita connection of a Riemannian metric and $N$ is a ball subbundle in $TM$ with fixed radius. It was shown  that  uniqueness always holds  and for real-analytic $g$   and small enough radius ac-structure   exists on $N$.  
 These  notions and results were later extended
to   pseudo-Riemannian metrics and general Koszul connections in  \cite{Sz04}, \cite{Bi03}, showing that uniqueness and existence still holds without the compactness of $M$. Other type of generalizations of ac-structures were  proposed for Finsler metrics \cite{DK03}, magnetic flows \cite{HK15} and sub-Riemannian structures \cite{U18}.
For Riemannian metrics Lempert \cite{L92} proved  that  existence of an ac-structure implies    real-analyticity of the metric.
It seems  still to be open
whether such regularity result holds   for general Koszul connections.

Ac-structures were further generalized in \cite{LSz12} and \cite{L19a}. Before we can recall these notions we need to discuss some generalizations of complex manifolds and holomorphic maps. 
 \begin{definition}
 An {\it involutive structure}  on a smooth manifold $X$
 is a smooth, involutive, complex subbundle $P$ of $\mathbb CTX$.  Involutivity means that the Lie bracket of local smooth sections of  $P$  is again
 a section of $P$. The pair $(X,P)$ is called an {\it involutive manifold}. If $(S,Q)$ is another involutive manifold, a smooth map $\phi:S\ra X$
 is involutive if $\phi_*Q\subset P$.
 An involutive structure 
 is called a {\it polarization}   if $\dim X$
 is even and
 rank$P=\frac1{2}\dim X$.  
\end{definition}
 \begin{definition}
 An  {\it adapted involutive  structure} (shortly an ai-structure) on $N$ is an involutive structure $P$  with
 ${\Omega_v}_*(T^{1,0}(\Omega_v^{-1}(N))\subset P$ for  $\forall v\in TM$. The ai-structure $P$ is an  {\it adapted complex polarization} (or ac-polarization) if $P$ is a 
 polarization.    
 $\nabla$ is said to be {\it polarization complete} or simply $\mathcal P$-{\it complete} if its ac-polarization  exists on $TM$.   
  A pseudo-Riemannian manifold $(M,g)$ is said to be $\mathcal P$-{\it complete}  if its Levi-Civita connection is $\mathcal P$-{\it complete}.
\end{definition}
For an ac-structure  on $N$, $P=T^{1,0}N$ will be an ac-polarization.
\subsection{Motivations}\label{Ss:motiv}
In this Subsection let  $(M,\nabla)$, $N$ and $\gamma_v$ be as in Subsection~\ref{Ss:notations}  and denote by $Z$  a complex manifold.
We would like to understand 
an ac-structure (resp. ac-polarization) on $N$  in terms of 
 holomorphic $\psi:N\ra Z$, (resp. involutive
 $\psi:(N,P)\ra T^{1,0}Z$) maps. 
 For a smooth map $f:M\ra Z$, $v\in T_pM$ and $\nabla$-geodesic $\gamma_v$, let $\chi_v:=f\circ\gamma_v$. Define
 \begin{equation}\label{E:Ni} 
N_i(f,Z):=\{v\in TM : \chi_v\text{ extends holomorpically to a neighborhood of } [0,i]\},
 \end{equation}
 and  denote by $\chi^\bC_v $  the holomorphic extension of $\chi_v$ (the actual  domain where $\chi_v^\bC$ lives plays no role here).
\begin{prop}\label{P:ac-hol-map}
 Suppose that an
  ac-structure  (resp. an ac-polarization $P$) exists on $N$ and
    let $\psi:N\ra Z$ be a holomorphic (resp. involutive) map. Denote by $f:=\left.\psi\right|_M$ its restriction. Then 
    $N\subset N_i(f,Z)$ and
    \begin{equation}
        \psi(v)=\chi^\bC_v(i)\qquad v\in N.
    \end{equation}
  In particular $f$ determines $\psi$ uniquely.  
\end{prop}
\begin{proof}
 Since $\Omega_v$ is holomorphic (involutive) on $\Omega_v^{-1}(N)$, so $\psi\circ\Omega_v$
is holomorphic as well. For every $\sigma\in I_p$
$$
\psi\circ\Omega_v(\sigma)=\psi(0_p\cdot\dot\gamma_v(\sigma))=f(\gamma_v(\sigma))=\chi_v(\sigma).
$$
Thus $\psi\circ\Omega_v$ is the holomorphic extension  $\chi^\bC_v$.  Since $\Omega_v(i)=\dot\gamma_v(0)=v$ and $N$ is a fiberwise $*$-domain in $TM$, $v\in N$ implies 
$[0,i]\subset \Omega_v^{-1}(N)$ and
\begin{equation}\label{E:psiomega}
\psi(v)=\psi(1\cdot\dot\gamma_v(0))=\psi\circ\Omega_v(i)=\chi^\bC_v(i).
\end{equation}
\end{proof}
This inspires the following definition.
\begin{definition}\label{D:polarmap}
  Let  $f:M\ra Z$ be  smooth. The corresponding {\it polar map} is defined as
  \begin{equation}\label{E:polarmap}
      \psi_f:N_i(f,Z)\ra Z,\qquad v\mapsto
      \chi^\bC_v(i).
  \end{equation}
\end{definition}
\begin{prop}\label{P:adap} 
Let $(M,\nabla)$ and $f:M\ra Z$ be real-analytic. 
Then 
\begin{enumerate}
    \item[a)]    $\psi_f\circ\Omega_v\in\mathcal O(\Omega_v^{-1}(N),Z)$ for  $\forall v\in TM$.
\item[b)] 
 If $\dim_\bR N=2\dim_\bC Z$ and $\psi_f:N\ra \psi_f(N)$ is a (local) diffeomorphism, then ac-structure exists on $N$.
\item[c)] 
 Let $P:=(\psi_f)_*^{-1}(T^{1,0}(Z))\subset\bC TN$. If $P$ is a smooth, involutive subbundle, then $P$ is an ai-structure on $N$. 
\end{enumerate}
\end{prop}
\begin{proof} (a)
  Let $v\in TM$ and $z:=\sigma+i\tau\in\Omega_v^{-1}(N)$ be arbitrary but fixed  and $w:=\Omega_v(z)\in N$.
First we show that $\chi_v$ has a holomorphic extension  $\chi^\bC_v:\Omega_v^{-1}(N)\ra Z$. Since $N$ is  a fiberwise $*$-domain in $TM$, consequently
$\Omega_v^{-1}(N)$ will be a fiberwise $*$-domain in $\bC I_p$,
it suffices to show that $\chi_v$ extends holomorphically to a neigborhood of  $[\sigma,\sigma+i\tau]$. For $\tau=0$ this  is indeed true, since the connection and $f$ are real-analytic. Now  assume $\tau\not=0$. The map $\varphi(s):=\tau s+\sigma$, $s\in\bC$ maps $[0,i]$ onto $[\sigma,\sigma+i\tau]$, so the same holds for small open neighborhoods. By our assumption $f\circ\gamma_w$ extends holomorphically to a neighborhood of $[0,i]$. On the other hand $\gamma_w(t)=\gamma_v(\tau t+\sigma)=\gamma_v(\varphi(t))$. Therefore $f\circ\gamma_w\circ\varphi^{-1}=f\circ\gamma_v$ extends holomorphically to a neighborhood of $[\sigma,\sigma+i\tau]$. That proves the holomorphic extendability of $\chi^\bC_v$ to  $\Omega_v^{-1}(N)$.

  Now $\gamma_w=\gamma_v\circ\varphi$ implies $\chi_w=\chi_v\circ\varphi$ and so $\chi^\bC_w(\zeta)=\chi^\bC_v(\varphi(\zeta))$, where $\zeta$ is in some open neighborhood of $[0,i]$. In particular

$\psi_f(\Omega_v(z))=\psi_f(\tau\dot\gamma_v(\sigma))=\chi^\bC_w(i)=\chi^\bC_v(\varphi(i))=\chi^\bC_v(\tau i+\sigma)=\chi^\bC_v(z)$.
This finishes the proof of part (a). Part (b) and (c) follows from (a) and the definitions.
\end{proof}
\begin{remarks}
Proposition \ref{P:adap} gives a hint how to construct an ac- or ai-structure on a given $N\subset TM$: choose  an appropriate complex manifold and real-analytic map $f:M\ra Z$ so that $N\subset N_i(f,Z)$.
Essentially part (a) and (b) has been  the method to prove the existence of an ac-structure
in a variety of situations: \cite{Sz91}, \cite{Sz95}, \cite{Sz98}, \cite{Bi03},   \cite{HI03}, \cite{Sz04},  \cite{HI06}, \cite{HI09}, \cite{ABI18}.

The method of part (c) is the  new idea of Lempert (\cite{L19a}) that is: even if  the polar map  has  singular points (i.e. not a local diffeomorphism), or the dimension of $Z$ is not the right one, the set theoretical inverse image of the bundle $T^{1,0}Z$ can produce (in good cases) an ai-structure.  
    In part (c) if we assume only
that $P$ is a genuine smooth vector  bundle, this  already implies the  involutivity of $P$.
This is a special case of  (\cite{L19a}, Lemma 3.1, (cf. also Theorem 3.4 in \cite{L19a}), since the bundle 
$T^{1,0}Z)$ is involutive. To find an appropriate $Z$ and  $f$ however is a nontrivial task.
 With the help of the polar map,  $P$
can always be defined for arbitrary $Z$ and $f$. But in general such $P$   will    be  a "bundle" only in some weak sense. It may not be locally trivial. Although each fiber $P_v$  is  a linear subspace
in $\mathbb CT_v(N)$, $v\in N$,  its dimension may vary if the rank of $\psi_*$  depends on $v$.  
Nevertheless such "generalized bundles" are studied in the literature, see for example \cite{Le11}.
Assuming that a notion of involutivity exists for such generalized objects, it seems reasonable to expect that our $P$ will have it and then  $P$ would be an ai-structure in some weak sense. We do not pursue this line of investigation in this paper.

In the paper  \cite{L19a}  Lempert was working in a more general situation then ours. He was interested in uniqueness and local existence of ai- (and more general) structures and worked with certain families of left-invariant involutive structures on monoids and their corresponding families of ai-structures on some open sets $N$ of trajectories of a vector field. The size of $N$ was not important for him, only its existence. $Z$ and $f$ were fixed. He considered sets of trajectories whose  complexifications    exist in a neighborhood
of a fixed compact set $0\in C\subset\bC $  with nonempty interior that is invariant w.r.t. the monoid action. We are interested in choosing $Z$ and $f$ so that  $N$ is as large as possible. Instead of a whole family, we work with one fixed polarization, the canonical complex structure on the complex plane and no monoid action. So for us it suffices to consider geodesics with weaker holomorphic extension, namely extension  to 
a neighborhood of $[0,i]$ as in \eqref{E:Ni}. But the basic idea is the same: with a good choice of
$Z$ and $f:M\ra Z$ to guarantee that the $P$ obtained  is a genuine smooth vector bundle.
 This is achieved in two situations: for general complexifications in Theorem~\ref{T:polarization} and  Theorem~\ref{T:icompl} and for left invariant Koszul connections on Lie groups in Theorem~\ref{T:eta-polarization} and Theorem~\ref{T:nabla-i-complete}. Before these investigations however we also need a general procedure that produces involutive structures. This is the content of the next Section. 
\end{remarks}
\section{Involutive structures on totally real submanifolds}
 Let $X$ be a complex manifold and  $N\subset X$ a smooth, 
    maximally totally real submanifold. Denote by $\mathbb CTN$ the complexified tangent bundle
  and $J:TX\ra TX$  the almost complex tensor. Let $p\in N$.
  The map
  $$v\mapsto v\qquad iw\mapsto -Jw,\qquad v,w\in T_pN$$
defines a $\bC$-linear isomorphism $l:\mathbb CTN\ra\left.(TX\right|_N,-J)$ between these 
  complex vector bundles.
Suppose  $E\ra X$ is an involutive, holomorphic
  subbundle of $T^{1,0}X$. Then $E$ defines a holomorphic foliation on $X$.
Denote by $L_p$ the leaf through a point $p\in X$. The tangent spaces $T_pL_p$ define the $J$ invariant involutive subbundle $V\ra X$ of $TX$ and $V^{1,0}=E$. 
\begin{theorem}\label{T:pol} 
Let $P:=l^{-1}(\left.V\right|_N)$. Then

    (a) $P=\mathbb CTN\cap(\left. T^{1,0}X\right|_N\oplus \left.V^{0,1}\right|_N)$,

    (b) $P\ra N$ is a complex, involutive subbundle of $\mathbb CTN$ and rank$P$= rank$E$,

    (c) $P$ defines a CR structure  on the open  subset
    $G=\{p\in N\mid T_pN\cap V_p=\{0\}\}\subset N$ ($G$ could be empty). When 
    $\dim_{\mathbb C}X=2$rank$E$,  $P$ yields a complex manifold structure on $G$.
  \end{theorem}
  \begin{proof}
 (a)  Denote the right hand side at $p$ by $H_p$. First we show that $V_p\subset l(H_p)$. 
Let $\xi\in V_p$.  Since $T_pN\subset T_pX$ is maximally totally real, $\exists! v,w\in T_pN$ with $\xi=v-Jw$. Then 
$l(v+iw)=\xi$ and
$$
v+iw=\frac{v+Jw}2-i\frac{J(v+Jw)}{2}+
\frac{v-Jw}{2}+i\frac{J(v-Jw)}{2}=\eta^{1,0}+\xi^{0,1}\in T^{1,0}_pX\oplus V^{0,1}_p,
$$
 proving 
$V_p\subset l(H_p)$.

Let now 
$v+iw\in H_p$. Thus $v,w\in T_pN$ and there exists 
$\eta\in T_pX$ and  $\xi\in V_p$ with
$$
v+iw=\eta^{1,0}+\xi^{0,1}=\frac{\eta-iJ\eta}{2}+\frac{\xi+iJ\xi}{2}.
$$
But then
$$
l(v+iw)=v-Jw=\frac{\eta+\xi}{2}-\frac{\eta-\xi}{2}=\xi\in V_p,
$$
 showing $l(H_p)\subset V_p$.

(b) That $P\ra N$ is a smooth complex subbundle of
$\mathbb CTN$ is obvious from its definition.
 To show its involutive as well, observe that the bundle $T^{1,0}X\oplus V^{0,1}$ is
 involutive. This can be seen using local trivializations, the integrability of the complex structure of $X$ and the involutivity of $E$.
  Now let  $\xi,\eta$ be local smooth sections of $P$. Extend them smoothly in a
  small neighborhood to get sections of the bundle $T^{1,0}X\oplus V^{0,1}$ and denote the extensions with the same symbols. $[\xi,\eta]$ restricted
  to $N$ will be a section of $\mathbb CTN$, since $N$ is a submanifold in $X$.
  But it is also a section of $T^{1,0}X\oplus V^{0,1}$ as well, since the latter is involutive.
  Hence $\left.[\xi,\eta]\right|_N$ will be a section of $P$ showing the involutivity of $P$.
  
 (c) It is enough to prove 
$$
P_p\cap\overline{P_p}=\mathbb C(T_pN\cap V_p),\qquad p\in N.
$$
Let $\eta=v+iw\in P_p\cap\overline{P_p}$. From the definition of $P$ we have $l(v+iw)=v-Jw\in V_p$ and $l(v-iw)=v+Jw\in V_p$. Hence $v, Jw\in V_p$. But $V_p$ is $J$-invariant, so this implies $w\in V_p$. Hence
  $P_p\cap\overline{P_p}\subset\mathbb C(T_pN\cap C_p)$.

Now let $v+iw\in \mathbb C(T_pN\cap V_p)$. Then $v,w\in T_pN\cap V_p$ and  using again the fact that $V_p$ is $J$ invariant,   $Jv, Jw\in V_p$ as well. Therefore
  $l(v+iw)=v-Jw\in V_p$ and $l(v-iw)=v+Jw\in V_p$. Now from the definition of $P$
  we get $v+iw, v-iw\in P_p$, showing 
  $P_p\cap\overline{P_p}\supset\mathbb C(T_pN\cap T_p(L_p))$.
\end{proof}  
\begin{definition}Let $X$ be a complex manifold, $N$ a smooth manifold and $\varphi:N\ra X$  an immersion. We shall say that $\varphi$ is {\it maximally totally real}, if for every $p\in N$, the subspace $\varphi_*(T_pN)$ is totally real in $T_{\varphi(p)}X$ and $\dim_\bR N=\dim_\bC X$.    
\end{definition}
\begin{theorem}\label{T:subm}
    Let $X^{m+k}, Z^k$ be  complex manifolds of dimension $m+k$ (resp. $k$), $\Psi:X\ra Z$ a holomorphic submersion, $N$  a  smooth manifold and $\varphi:N\ra X$ a maximally totally real immersion. Let $\psi:=\Psi\circ\varphi$ and $\psi_*:\bC TN\ra \bC TZ$ the induced map. Then
  $P:= \psi_*^{-1}(T^{1,0}Z)\subset \bC TN$ defines an involutive structure on $N$ of rank $m$.  
\end{theorem}
\begin{proof}
  Let $\Psi_*:TX\ra TZ$ be the induced map and $V:=\ker\Psi_*$.  The rank-theorem 
yields that $V\ra X$ is an involutive  vector subbundle of $TX$. Since $\Psi$ is holomorphic, $V$ is $J$-invariant and $E:=V^{1,0}$ is an involutive holomorphic subbundle of $T^{1,0}X$. For any $p\in N$, the map $\varphi$ will be a maximally totally real  embedding of a small open neighborhood of $p$ into a small open neighborhood of $\varphi(p)$ in $X$. Restricting everything to these neighborhoods we can assume that in fact $N$ is a submanifold of $X$. So we are in the situtation of Theorem \ref{T:pol}. 
Let $\widetilde P:=l^{-1}(\left.V\right|_N)$.  We need to show that $\widetilde P=P$. 
Let $v+iw\in \widetilde P_p$, $p\in N$. From Theorem \ref{T:pol} part (a) we know that there exists $\xi\in T_pX$ and $\eta\in V_p$ so that $v+iw=\xi^{1,0}+\eta^{0,1}$. Then 
$\psi_*(v+iw)=\Psi_*(\xi^{1,0})+\Psi_*(\eta^{0,1})=\Psi_*(\xi^{1,0})\in T^{1,0}Z$ showing $\widetilde P_p\subset P_p$.
Now let $v+iv\in P_p$. Then $\psi_*(v+iw)=\Psi_*v+i\Psi_*w\in T^{1,0}Z$ implies
$J_Z\Psi_*v=-\Psi_*w$. Since $\Psi$ is holomorphic we get $\xi:=w+J_Xv\in \ker \left.\Psi_*\right|_p=V_p$. Therefore
$$
v+iw=v-iJ_Xv+i(J_Xv+w)=2v^{1,0}+i\xi^{1,0}+i\xi^{0,1}\in \bC T_pN\cap(T^{1,0}_pX\oplus V^{0,1}_p)=P_p
$$ 
\end{proof} 
\begin{remark}
 Similar ideas, that were used in the proof above, already appeared in \cite{L19a} Proposition 4.3.  
\end{remark}
 \section{Flows}
 \subsection{Flow of a vector field}\label{Ss:flow}
 Here we fix some notation and collect some well known facts. As a reference see for example \cite{Bo86}.
 Let $X$ be a smooth (i.e. $C^\infty$) or real-analytic manifold and $V$ a vector field on $X$ with the same smoothness. For $p\in X$, let $I_p=(\alpha_p,\beta_p)$, where $\alpha_p<0<\beta_p$, be the maximal interval where an integral curve $c_p:I_p\ra X$ of $V$ with initial condition $c_p(0)=p$ exists. $c_p$ is called {\it complete}, if $I_p=\bR$. $V$ is said to be $\bR$-{\it complete} if all of its  integral curves are complete.
 The set
 $$
 \mathcal D=\bigcup_{p\in X} I_p\times\{p\}
 $$
 is an open subset in $\bR\times X$.
 For a  $t\in\bR$ let
 $$
 D_t:=\{p\in X : t\in I_p\}.
 $$
Then $\mathcal D\cap(\{t\}\times X)=\{t\}\times D_t$, $D_t$ is open (possible empty) in $X$. The (local) flow of $V$ is the smooth (resp. real-analytic) map
$$
\Phi:\mathcal D\ra X\qquad (t,p)\mapsto    \Phi(t,p):=c_p(t).
$$
For fixed $t\in\bR$ and $D_t\not=\emptyset$, the {\it time-t} map
is $\phi^V_t:=\Phi(t,.):D_t\ra X$. If it is clear from the content we just write $\phi_t$. Then
$\phi_t(D_t)=D_{-t}$ and $\phi_{-t}\circ\phi_t(p)=p$ for all $p\in D_t$.
Hence $\phi_t(D_t)$ is also open in $X$ and $\phi_t$ is a diffeomorphism between $D_t$ and its image.

Let now $X$ be a complex manifold, $\mathcal V$ a holomorphic vector field on $X$ and $V=2$Re$\mathcal V$. Then for $t\in \bR$, the time-t map $\phi^V_t$ will be biholomorphic between $D_t$ and its image. 
\begin{definition}
      $D^{\mathcal V}_i$ denotes the set of those $p\in X$ for which the holomorphic trajectory $\chi_p$ of $\mathcal V$ with $\chi_p(0)=p$ exists in an open neighborhood of $[0,i]$ (the neighborhood can depend on $p)$. We call the map $\phi^{\mathcal V}_i:D^{\mathcal V}_i\ra X$ defined by
$\phi^{\mathcal V}_i(p):=\chi_p(i)$ the {\it time-i} map of $\mathcal V$.
\end{definition}
\begin{prop}\label{P:flowdomain} Let $X$ be a complex manifold, $\mathcal V$ a holomorphic vector field on $X$ and  $V:=2\Re\mathcal V$.  Then 
$D^{\mathcal V}_i=D^{JV}_1$ and $\phi^\mathcal V_i(p)=\phi^{JV}_1(p)$, $p\in D^{\mathcal V}_i$. In particular $D^{\mathcal V}_i$ and $\phi^\mathcal V_i(D^{\mathcal V}_i)$ are  open in $X$ and $\phi^\mathcal V_i:D^{\mathcal V}_i\ra
\phi^\mathcal V_i(D^{\mathcal V}_i)$ is a biholomorphism.    
\end{prop}
\begin{proof} Let $p\in D^{\mathcal V}_i$
    and $\chi_p$ the corresponding holomorphic integral curve of $\mathcal V$.
    Then  $d(t):=\chi_p(it)$ will be an integral curve of $JV$ (where $J$ is the almost complex tensor on $X$) defined in a neighborhood of $[0,1]$ with $d(0)=p$. Hence $p\in D^{JV}_1$ and $\phi^{JV}_1(p)=d(1)=\chi_p(i)=\phi^{\mathcal V}_i(p)$. 
    
    Now let $p\in D^{JV}_1$ and $d:I_p\ra X$  the integral curve of $JV$ with $d(0)=p$ and $1\in I_p$. Since $V$ is real-analytic, $d$  extends holomorphically to a neighborhood of $[0,1]$ and   
    $\chi(\zeta):=d(-i\zeta)$ will be a holomorphic trajectory of $\mathcal V$ defined in some neighborhood of $[0,i]$. 
    Thus $D^{\mathcal V}_i=D^{JV}_1$ and $\phi^\mathcal V_i(p)=\phi^{JV}_1(p)$. Hence $D^{\mathcal V}_i$ and $\phi^\mathcal V_i(D^{\mathcal V}_i)$ are open in $X$ and $\phi^\mathcal V_i=\phi^{JV}_1:D^{JV}_1\ra \phi^{JV}_1(D^{JV}_1)$ is a diffeomorphism
     (c.f. \cite{Bo86}).
    Since $\mathcal V$ is holomorphic, $\phi^\mathcal V_i=\phi^{JV}_1$ will be  biholomorphic (c.f. \cite{Hi97}).
\end{proof}
\begin{definition}\label{D-Ccomplete}
The holomorphic vector field $\mathcal V$ on the complex manifold $X$
is called {\it $\bC$-complete } if every holomorphic trajectory is defined over $\bC$.
\end{definition}
 \subsection{Geodesic flows}\label{Ss:geodflow}
 Let $Z$ be a complex manifold, $T^{1,0}Z$ its holomorphic tangent bundle and $\pi:T^{1,0}Z\ra Z$
 the bundle projection. Let $\mathfrak D$ be a holomorphic Koszul connection on $T^{1,0}Z$, i.e. a $\bC$-linear map (of sheaves) $\mathfrak D:T^{1,0}Z\ra\Omega_Z\otimes T^{1,0}Z$ which satisfies the Leibniz rule
 \begin{equation}
     \mathfrak D(f\cdot s)=\partial f\otimes s+f\cdot\mathfrak D s,
 \end{equation}
 where $\Omega_Z$ denotes the holomorphic cotangent sheaf, $f$ any local holomorphic function on $Z$ and $s$ any local holomorphic vector field on $Z$.
 For a $w\in T^{1,0}Z$, let $\Gamma_w$ be the holomorphic $\mathfrak D$-geodesic defined in some open neighborhood $G$ of $0$ with initial condition $\dot\Gamma^{1,0}_w(0)=w$ and define $\chi_w:G\ra T^{1,0}Z$, with $\zeta\mapsto\dot\Gamma^{1,0}_w(\zeta)$ and finally let $\mathcal V(w):=\dot\chi^{1,0}_w(0)$.  Then $\mathcal V :X\ra T^{1,0}X$ is a holomorphic (1,0) vector field on the complex manifold $X:=T^{1,0}Z$, the total space of the holomorphic tangent bundle.
 \begin{definition}
     $\mathcal V$ is called the {\it geodesic spray} of $\mathfrak D$ 
     (c.f. \cite{La95} for the real version). The holomorphic Koszul manifold $(Z,\mathfrak D)$ is called {\it $\bC$-complete} if all the holomorphic geodesics are defined over $\bC$ or equivalently if $\mathcal V$ is $\bC$-complete.
 \end{definition}
By its definition the  holomorphic integral curves of $\mathcal V$ are the tangential curves of holomorphic $\mathfrak D$-geodesics. Let $V:=2\Re \mathcal V$ and define
 \begin{equation}\label{E:Gi}
 G_i:=\{w\in T^{1,0}Z : \Gamma_w\text{ exists in a small open neighborhood of } [0,i]
 \}.
 \end{equation}
 \begin{prop}\label{P:Giopen}
 The set $ G_i\subset T^{1,0}Z$ is open, $Z\subset G_i$ and for every $p\in Z$, $G_i\cap T^{1,0}_pZ$ is a  $*$-domain in $T^{1,0}_pZ$ w.r.t. $0_p$. The map $\Psi:G_i\ra Z$, $w\mapsto \Gamma_w(i)=\pi\circ\phi^{JV}_1(w)$ is a holomorphic submersion.     
 \end{prop}
    \begin{proof}
        What was said above together with Proposition \ref{P:flowdomain} yields
        $G_i=D^\mathcal V_i=D^{JV}_1\subset X=T^{1,0}Z$, hence $G_i$ is open. The points of $Z$ correspond to the constant geodesics, so $Z\subset G_i$. If $w\in G_i$ and $0\le \tau\le 1$, then $c(\zeta):=\Gamma_w(\tau\zeta)$ is the holomorphic geodesic with $\dot c^{1,0}(0)=\tau w$, showing $\tau w\in G_i$. Also from Proposition \ref{P:flowdomain} we  get, that 
        the map 
        $\phi^\mathcal V_i:G_i\ra T^{1,0}Z$, $\phi^\mathcal V_i(w)=\chi_w(i)=\dot\Gamma^{1,0}_w(i)$ is a biholomorphism to its image and
$\pi\circ\phi^\mathcal V_i(w)=\pi(\dot\Gamma^{1,0}_w(i))=\Gamma_w(i)=\Psi(w)$. Therefore $\psi$ is a holomorphic submersion. 
    \end{proof}
The connection $\mathfrak D$ induces an ordinary  Koszul connection $\widehat\nabla$ on $TZ$ as follows. The $\bC$-bundle isomorphism  $\xi:TZ\ra T^{1,0}Z, u\mapsto u^{1,0}$ yields a complex manifold structure on the smooth manifold $TZ$ by pulling back the complex manifold structure of the total space $T^{1,0}Z$. The pulled back vector field $\xi^*(V)$ will be a real vector field on the total space  $TZ$, the geodesic spray of $\widehat\nabla$. Looking at  this from a different angle:  $\mathfrak D+\bar{\partial}$ 
will be an ordinary Koszul connection on the complex vector bundle $T^{1,0}Z$ that yields via the  $\bC$-bundle isomorphism $\xi$ the Koszul connection $\widehat \nabla$ on $TZ$.
One can show that for  any local holomorphic (1,0) vector field $u$ in $Z$, 
$\widehat\nabla (u+\bar u)=\mathfrak Du+\overline{\mathfrak Du}$ and in fact   this formula could be used to define $\widehat \nabla$ . What is important for us is the following property of $\widehat\nabla$ (c.f. \cite{LB80, LB83}):  
\begin{equation}\label{E:geod}
\widehat\nabla\text{-geodesics are precisely the restrictions of holomorphic } \mathfrak D\text{-geodesics to } \mathbb R.
\end{equation}
\section{Complexifications}\label{S:complexification}
\begin{definition}\label{D:complexification}
 Let $M$ be a real-analytic manifold  with a  real-analytic  Koszul connection $\nabla$. A {\it complexification} of $(M,\nabla)$ is a triple $(Z,\mathfrak D, \iota)$, where $Z$ is a complex manifold, $\mathfrak D$  a holomorphic Koszul connection on $T^{1,0}Z$, $\widehat\nabla$ the corresponding real Koszul connection on $TZ$ and 
 $\iota:M\ra Z$ a maximally totally real, real-analytic immersion with 
 $\nabla=\iota^*(\widehat\nabla)$.  
\end{definition}
For  $v\in T_pM$ let $\gamma_v:I_p\ra M$ be the $\nabla$-geodesic with initial condition $\dot\gamma_v(0)=v$  and $\chi_v:=\iota\circ\gamma_v$.
 Recall from Section~\ref{Ss:motiv} the definition of the set $N_i:=N_i(\iota, Z)\subset TM$ (see \eqref{E:Ni}) and the corresponding polar map $\psi_\iota:N_i\ra Z$ (see \eqref{E:polarmap}).
\begin{theorem}\label{T:polarization} Let
$(M,\nabla)$ be a real-analytic Koszul manifold 
with complexification 
 $(Z,\mathfrak D,\iota)$. Then $N_i\subset TM$  is an open, fiberwise $*$-domain  (w.r.t. $o_p \in T_pM$).   The polar map  $\psi_\iota:N_i\ra Z$ is real-analytic and with the induced map
 $(\psi_\iota)_*:\bC TN_i\ra \bC TZ$, 
$P:=(\psi_\iota)^{-1}_*(T^{1,0}Z)$
defines an ac-polarization on $N_i$. Let $\mathcal S\subset N_i$ be the set of those points where $\psi_\iota$ is not a local diffeomorphism. Then $\mathcal S$ is 
a real-analytic subset in $N_i$ with empty interior and 
the polar map defines  a complex manifold structure on $D:=N_i\setminus\mathcal S$. Denote by $D_0$ the connected component of $D$ that contains $M$. Then the complex structure on $D_0$  will be an ac-structure.
\end{theorem}
\begin{proof}
   For $v=0_p\in T_pM$, the constant map $\gamma(t)=p$ shows $M\subset N_i$. Since $\nabla=\iota^*(\widehat\nabla)$,  for a $v\in TM$ the curve $\chi_v=\iota\circ\gamma_v$ will be a $\widehat\nabla$-geodesic with initial condition $\dot\chi_v(0)=\iota_*v$. From \eqref{E:geod} we know that the holomorphic $\mathfrak D$-geodesic $\Gamma_w$ with $w=(\iota_*v)^{1,0}$ extends $\chi_v$  holomorphically  to  a neighborhood of $0$. That is
   $v\in N_i$ iff $(\iota_*v)^{1,0}\in G_i$ (where $G_i$ is from \eqref{E:Gi}). 
   Using  $\xi:TZ\ra T^{1,0}Z$, $w\mapsto w^{1,0}$, this means $N_i=(\xi\circ\iota_*)^{-1}(G_i)$ and 
   $N_i\cap T_pM=(\xi\circ\iota_*)^{-1}(G_i\cap T^{1,0}_{\iota(p)}Z)$.
   Now Proposition \ref{P:Giopen} implies  $G_i$ is open in $T^{1,0}Z$, hence $N_i$ will be open in $TM$ and that $N_i\cap T_pM$ is a $*$-domain in $T_pM$.

    Since  $(Z,\nabla^{\mathbb C},\iota)$ is a complexification of $(M,\nabla)$, the map $\xi\circ\iota_*:N_i\ra G_i$ will  be  a maximally totally real immersion. From Proposition \ref{P:Giopen} the map  $\Psi:G_i\ra Z$, $w\mapsto\Gamma_w(0)$ is a holomorphic submersion and $\psi=\Psi\circ\xi\circ\iota_*$. Then Theorem \ref{T:subm} shows that $P$ is a smooth involutive bundle.
    In light of Proposition \ref{P:adap}, $P$ is an ac-polarization.

$\mathcal S$ is the zero set of the  section $\Lambda^{2n}(\psi_\iota)_*$ of the line bundle
$Hom(\Lambda^{2n} TN_i,\Lambda^{2n} TZ)$, where
$n=\dim_\bR M$. 
     This section is real-analytic, since the polar map is real-analytic. Also the polar map is known to be locally biholomorphic in the points of $M$, see for example Lemma 3.1 in \cite{Sz04}. Therefore 
     $\Lambda^{2n}(\psi_\iota)_*$ is not identically zero. Hence $\mathcal S$ is a proper subset of $TM$, so by real-analyticity, its interior must be empty. Since $\psi_\iota$ is a local diffeomorphism on $D_0$, the ac-polarization $P$ is actually a genuine  ac-structure on $D_0$.
   \end{proof}
    \begin{definition}\label{D:icompl-complex}
    The   complexification 
    $(Z,\mathfrak D,\iota)$ of $(M,\nabla)$
    is called {\it $\bC$-complete}, if 
    for every $v\in TM$, $\chi_v$ extends holomorphically to $\bC$; it is called {\it $i\bR$-complete} if for each $v$, $\chi_v$ extends holomorphically to a neighborhood of $i\bR$ or equivalently if $N_i(\iota,Z)=TM$.
 \end{definition} 
As a corollary of Theorem \ref{T:polarization} we get.
\begin{theorem}\label{T:icompl}
    Let $(M,\nabla)$ be a real-analytic Koszul manifold and  $(Z,\mathfrak D,\iota)$  an $i\bR$-complete  (or $\bC$-complete) complexification. Then 
    $P:=(\psi_\iota)^{-1}_*(T^{1,0}Z)$
defines an ac-polarization on $TM$ i.e.
    $\nabla$ is $\mathcal P$-complete.
\end{theorem}
\section{Left invariant connections}
\subsection{Left invariant connections and bilinear maps}\label{Ss:canonical}
Let $G$ be a real Lie group with unit element $e$ and Lie algebra  $\mathfrak g=T_eG$.
 Left invariant Koszul connections $\nabla$ on $TG$ and   $\mathbb R$-bilinear maps 
$\alpha:\mathfrak g\times\mathfrak g\ra\mathfrak g$ 
bijectively correspond to each other 
 by  
$\alpha(x,y):=(\nabla_xy)_e$, where
$x,  y\in T_eG$  (the same letters denote the left invariant vector field extensions as well). For a given $\alpha$,  $\overset{\alpha}\nabla$ denotes the corresponding connection. The vector field $\mathcal W$ on $\mathfrak g$, defined by the quadratic homogeneous polynomial map
 $\mathfrak g\ra \mathfrak g$, $x\mapsto -\alpha(x,x)$ 
 is called the {\it Euler-Arnold vector field}  of $\nabla$.

For a connection $\nabla$, let $\alpha=\alpha'+\beta$ be the decomposition to antisymmetric and symmetric parts.  $\nabla$ is torsion free iff $\alpha'(x,y)=\frac1{2}[x,y]$. 
The unique torsion free connection  with $\beta=0$ is  the {\it canonical (or Cartan-0) connection}, denoted by $\overset{c}  {\nabla} $. 
For a given $\alpha=\alpha'+\beta$, let $\widetilde{\alpha}(x,y):=\frac1{2}[x,y]+\beta(x,y)$. Then the connection $\overset{\widetilde{\alpha}}\nabla$ will be torsion free and has the same geodesics with the same parametrization as $\overset{\alpha}\nabla$
(see \cite{Sp79}, Ch6. Add.1). Since the ac-structure only depends on the geodesics, we can always assume that the connection we study is torsion free.

Let now $H$ be a complex Lie group, $J:\mathfrak h\ra\mathfrak h$ the almost complex structure and
$T^{1,0}_eH=\mathfrak h^{1,0}:=\{x-iJx\mid x\in\mathfrak h\}$.  Similarly to the real case, left invariant, holomorphic Koszul connections $\mathfrak D$ on $T^{1,0}H$ and      complex bilinear maps
$
\alpha^{1,0}:\mathfrak h^{1,0}\times\mathfrak h^{1,0}\ra\mathfrak h^{1,0},
$
bijectively correspond to each other,
where $\alpha^{1,0}(u,v):=(\mathfrak D_ uv)_e$ with $u, v\in \mathfrak h^{1,0}$.
For a holomorphic connection $\mathfrak D$,  the induced  left invariant real connection $\widehat\nabla$  will correspond to the bilinear map $\alpha:\mathfrak h\times\mathfrak h\ra\mathfrak h$, 
$\alpha(x,y):=2\Re\alpha^{1,0}(x^{1,0},y^{1,0})$. The  holomorphic 
vector field $\mathcal E$  on 
$\mathfrak h^{1,0} $, defined by the quadratic complex homogeneous polynomial map  $u\mapsto -\alpha^{1,0}(u,u)$, $u\in\mathfrak h^{1,0}$ is called the  {\it holomorphic   Euler-Arnold vector field}.
Using the complex vector space isomorphism
$\xi:(\mathfrak h,J)\ra \mathfrak h^{1,0}$, $x\mapsto x^{1,0}$, we shall also call $\mathcal V:=\xi^*\mathcal E$ the holomorphic Euler-Arnold vector field of $\mathfrak D$ defined on 
the vector space $(\mathfrak h,J)$ as a 
 complex manifold. The Euler-Arnold  vector field of the real connection $\widehat\nabla$  will be 2$\Re\xi^*\mathcal V$.

For a connection $\mathfrak D$,  $\alpha^{1,0}=\alpha'_\bC+\beta_\bC$ denotes the decomposition to complex bilinear antisymmetric and complex bilinear symmetric parts.  $\mathfrak D$ is torsion free iff 
$\alpha'_\bC(u,v)=\frac1{2}[u,v]$. 
As it was said in the real case is valid here as well: we can assume that $\mathfrak D$ is torsion free. 
The unique holomorphic left invariant torsion free connection with $\beta_\bC=0$ is  the {\it canonical  holomorphic (or Cartan) connection} $\overset{c}{\mathfrak D}$. Its corresponding bilinear map is $\alpha^{1,0}(u,v)=\frac1{2}[u,v]$, $u,v\in \mathfrak h^{1,0}$. Since
$[x^{1,0},y^{1,0}]=[x,y]^{1,0}$, for $x,y\in\mathfrak h$, the real left invariant Koszul connection induced by  the holomorphic connection $\overset{c}{\mathfrak D}$ will be $\overset{c}\nabla_H$, the  canonical  connection of $H$ considered as a real Lie group.
\subsection{Geodesics of left invariant connections on real Lie groups}\label{Ss:leftinvreal}
In this Subsection $G$ denotes a real Lie group with Lie algebra $\mathfrak g$, $\nabla$  a left invariant Koszul connection on $TG$ and $\alpha:\mathfrak g\times\mathfrak g\ra\mathfrak g$ the 
 corresponding bilinear map.
\begin{prop}\label{P:leftinvconn}   
Let $\gamma:(a,b)\ra G$ be a smooth curve and define $x:(a,b)\ra\mathfrak g$ by 
$x(t):=(L_{\gamma(t)}^{-1})_*(\dot\gamma(t))$. Then $\gamma$ is a $\nabla$-geodesic iff
$$\dot x+\alpha(x,x)=0,$$
i.e. iff $x$ is an integral curve of the Euler-Arnold vector field.
\end{prop}
\begin{proof}
    Let $e_1,\dots,e_n$ be  a basis of $T_eG=\mathfrak g$ and $E_j$ their left invariant extensions. Then
    $$   x(t)=\sum\limits^n_{j=1}f_j(t)e_j,\quad\text{and}\quad \dot\gamma(t)=(L_{\gamma(t)})_*(x(t))=\sum\limits^n_{j=1}\left.f_j(t)E_j\right|_{\gamma(t)}.
    $$
    with smooth functions $f_j:(a,b)\ra \mathbb R$.
    Denote by $\frac{D}{dt}$ the covariant derivative operator acting on vector fields defined along $\gamma$.
    Then
    $$
    \frac{D\dot\gamma}{dt}(t)=\sum\limits^n_{j=1}f'_j(t)E_j+\sum\limits^n_{j=1}f_j(t)\nabla_{\dot\gamma(t)}E_j.
    $$
    Since $\nabla$ and $E_j$ are left invariant, we have  
    $$    
    \nabla_{\dot\gamma(t)}E_j=
    \nabla_{(L_{\gamma(t)})_*(x(t))}E_j
    =(L_{\gamma(t)})_*(\nabla_{x(t)}E_j)
    $$
    Hence
$$\frac{D\dot\gamma}{dt}(t)=(L_{\gamma(t)})_*\left(\sum\limits^n_{j=1}f'_j(t) e_j\right)+
    \sum\limits^n_{j=1}f_j(t)(L_{\gamma(t)})_*\nabla_{x(t)}E_j\\
    =
    (L_{\gamma(t)})_*(\dot x(t)+\alpha(x(t),x(t)).
$$    
\end{proof}
As a corollary of Proposition~\ref{P:leftinvconn} we get the following well known statement (c.f. \cite{A66},\cite{A78}).
\begin{cor}\label{C:EAequation} The Levi-Civita connection  of a left invariant pseudo-Riemannian metric  $g$ is
 \begin{equation}\label{E:LeviCivita}
  \nabla_XY=\frac1{2}\{[X,Y]-ad^*_XY-ad^*_YX\}, \quad X, Y\in\mathfrak g.   
 \end{equation}      
       Let $\gamma:(a,b)\ra G$ be a smooth curve and  define $x:(a,b)\ra\mathfrak g$  by 
$x(t):=(L_{\gamma(t)}^{-1})_*(\dot\gamma(t))$. Then
     $\gamma$ is a geodesic iff
$$
\dot x=ad^*_xx.
$$ 
 The Euler-Arnold vector field of $\nabla$ is $\mathcal W(x)=-\alpha(x,x)=ad^*_xx$.
\end{cor}
\begin{proof}
Let $X,Y,Z$ be arbitrary smooth vector fields on $G$.
Then Koszul's formula gives
\begin{equation}\label{E:nab}
\begin{split}
g(\nabla_XY,Z)=\frac1{2}\{Xg(Y,Z)-Zg(X,Y)+Yg(Z,X)+\\+g(Z,[X,Y])+g([Z,X],Y)+g(X,[Z,Y])\}.
\end{split}
\end{equation}
If $X,Y,Z$ are left invariant, the quantities $g(Y,Z), g(X,Y), g(Z,X)$ are all constant and formula~\eqref{E:nab} yields
\begin{equation}\label{E:nabla}
\nabla_XY=\frac1{2}\{[X,Y]-ad^*_XY-ad^*_YX\}.
\end{equation}
Now for fixed $t$ let $Y=X$ be the left invariant extension of $x(t)$.  From \eqref{E:nabla} we get 
$$
\alpha(x(t),x(t))=\left.\nabla_XX\right|_e=-\left.ad^*_XX\right|_e=-ad^*_ {x(t)} x(t)
$$
and Proposition~\ref{P:leftinvconn} finishes the proof.
\end{proof}
 Recall that the Maurer-Cartan form $\omega_G$  is the $\mathfrak g$ valued left invariant $1$-form on $TG$ defined by 
$$
\omega_G(v):=(L_c^{-1})_*v,\quad \text{where}\quad v\in T_cG.
$$
Hence the function $x(t)$ in Proposition~\ref{P:leftinvconn} can be written as $x(t)=\omega_G(\gamma'(t))$.
Theorem 3.7.1 and Theorem 3.5.2 in \cite{S00} yields the following.
\begin{prop}\label{P:fundthm} Let
$x:(a,b)\ra\mathfrak g$ be a smooth curve. Then there exists a smooth map $\gamma:(a,b)\ra G$  with 
\begin{equation}\label{E:MCeq}
    \omega_G(\gamma'(t))=x(t),\qquad t\in(a,b).
\end{equation}
    Furthermore if this holds for $\gamma_1,\gamma_2:(a,b)\ra G$, then  $\exists c\in G$, so that $\gamma_2=c\gamma_1$.
\end{prop}
When $G=GL(n,\bR)$ (or more generally any matrix Lie group), $\omega_G=a^{-1}da$, $a\in GL(n,\bR)$ and $\omega_G(\gamma'(t))=x(t)$ will be equivalent to $\gamma'=\gamma x$, that is a system of linear equations and the existence of $\gamma$ follows immediately.

Proposition~\ref{P:leftinvconn} and Proposition~\ref{P:fundthm} yields.
\begin{cor}\label{C:conncomplete} 
\
\begin{enumerate}
    \item[a)]
 Let $x_0\in\mathfrak g$, $t_0\in (a,b)$.  Then there exists a geodesic  $\gamma:(a,b)\ra G$ with $\gamma'(t_0)=x_0$ iff the system
\begin{equation}
   \dot x+\alpha(x,x)=0,\qquad x(0)=x_0.
   \end{equation}
   has a solution  defined on $(a,b)$.    
\item[b)] $\nabla$ is complete iff the corresponding Euler-Arnold vector field is complete.
\end{enumerate}
\end{cor}
When $\nabla$ is the Levi-Civita connection of a left invariant pseudo-Riemannian metric, Corollary~\ref{C:conncomplete}(b) was shown in \cite{AP90}.
From Corollary~\ref{C:conncomplete} we get.
\begin{cor} 
 Let  $\beta:\mathfrak g\times\mathfrak g\ra\mathfrak g$ be a symmetric bilinear 
map. Suppose  the vector field $\mathcal W$, corresponding to $x\mapsto-\beta(x,x)$, $x\in \mathfrak g$ is complete.
 Then the left-invariant torsion free Koszul connection  defined by 
\begin{equation}
 \nabla_xy:=\frac{1}{2}[x,y]+\beta(x,y)\qquad x,y\in\mathfrak g.   
\end{equation}
is complete.    
\end{cor}
\subsection{Holomorphic geodesics of left invariant holomorphic connections}\label{Ss:leftinvC}
In this subsection $H$ denotes a complex Lie group with Lie algebra $\mathfrak h$ and $\mathfrak D$ a left invariant holomorphic Koszul connection on $T^{1,0}H$ with corresponding $\bC$-bilinear form
$\alpha^{1,0}:\mathfrak h^{1,0}\times\mathfrak h^{1,0}\ra\mathfrak h^{1,0}$.
 With basically the same proof and straightforward changes in terminology,
the complex version of Proposition~\ref{P:leftinvconn} and Corollary~\ref{C:EAequation} are also true:
\begin{prop}\label{P:CEA}   Let $U\subset\bC$ be a domain and $\Gamma\in\mathcal O(U,H)$. Define 
 $f\in\mathcal O(U,h^{1,0})$  by $f(\zeta):=(L^{-1}_{\Gamma(\zeta)})_*(\Gamma'(\zeta)^{1,0})$. Then
\begin{enumerate}
    \item[a)] $\Gamma$ is a holomorphic $\mathfrak D$-geodesic iff $f'(\zeta)+\alpha^{1,0}(f(\zeta),f(\zeta))=0$.
\item[b)] Suppose $\mathfrak D$ is the holomorphic  Levi-Civita connection of the left invariant holomorphic Riemannian metric $q$. Then 
\begin{equation}
\mathfrak D_uv=\frac1{2}\{[u,v]-ad^*_uv-ad^*_vu\}, \quad u, v\in \mathfrak h^{1,0},
\end{equation}
and so
$\alpha^{1,0}(u,u)=-ad^*_uu$, where $*$ means  the adjoint w.r.t. the nonsingular complex bilinear form $q_e:T^{1,0}_eH\times T^{1,0}_eH\ra\bC$. Thus $\Gamma$ is a holomorphic $\mathcal D$-geodesic iff $f'=ad^*_ff$.
\end{enumerate}
\end{prop}
The Maurer-Cartan form of $H$  is the  left invariant, $\mathfrak h^{1,0}$ valued holomorphic form $\omega_H$  defined by $\omega_H(u)=(L_a^{-1})_*u$, $u\in T^{1,0}_aH$. Hence 
\begin{equation}\label{E:MCisom}
    \omega_H:T^{1,0}_aH\ra \mathfrak h^{1,0}
\end{equation} 
is a $\bC$-linear isomorphism for every $a\in H$.
\begin{prop}\label{P:CMC} Let  $D\subset\bC$ be a domain. 
\begin{enumerate}
    \item[a)] If $\Gamma_1,\Gamma_2\in\mathcal O(D,H)$ and $\omega_H(\Gamma'_1)\equiv\omega_H(\Gamma'_2)$, then $\exists c\in H$, so that $\Gamma_2\equiv c\Gamma_1$.
\item[b)] Suppose $D$ is simply connected, $\zeta_0\in D$, $a\in H$ and $u\in\mathcal O(D,\mathfrak h^{1,0})$. Then $\exists!$ $\Gamma:(D,\zeta_0)\ra (H,a)$ holomorphic with $\omega_H(\Gamma')\equiv u$.
\end{enumerate}
    \end{prop}
\begin{proof}
a) The proof of the real version of the statement (see \cite{S00}, Theorem 3.5.2) works in the complex case as well.

b) In light of part a) we only need to prove existence. First we show that we can solve the equation locally. 

Let $u_1,\dots, u_n\in T_e^{1,0}H$ be a $\bC$-basis and   $U_j$ the left invariant extension of $u_j$. 
Let $\zeta_0\in D$ and $b\in H$ be fixed.
Choose a local holomorphic system of coordinates $\psi:(V,b)\ra(\bC^n,0)$ in $H$ near $b$ and identify the points of $V$ with their image via $\psi$ in $\psi(V)=:G$. Denote the points of $G$   by $z$.
\begin{equation}   \omega_H(\left.\partial_{z_j}\right|_z)=\sum\limits^n_{k=1}a_{kj}(z)u_k
\end{equation} 
with some $a_{kj}\in\mathcal O(G,\bC)$. Due to \eqref{E:MCisom} the matrix $A(z)=(a_{kj}(z))$ will be invertible. 
Suppose $\Gamma:(W,\zeta_0)\ra (G,0)$ is  holomorphic, where $W\subset D$ is some open neighborhood of $\zeta_0$.  Then $\Gamma=(\gamma_1,\dots,\gamma_n)$ and
\begin{equation}    \omega_H(\Gamma'(\zeta))=\omega_H\left(\sum_j\gamma'_j(\zeta)\left.\partial_{z_j}\right|_{\Gamma(\zeta)}\right)=\sum_{j,k}\gamma'_j(\zeta)a_{kj}(\Gamma(\zeta))u_k.
\end{equation}
Therefore the equation
\begin{equation}
    \omega_H(\Gamma'(\zeta))=u(\zeta)=\sum b_k(\zeta)u_k, \quad \zeta\in W
\end{equation}
is equivalent to
$$b(\zeta)=A(\Gamma(\zeta))\Gamma'(\zeta).$$
Since $A$ is invertible, this latter equation can be rewritten as
$$
\Gamma'(\zeta)=A^{-1}(\Gamma(\zeta))b(\zeta).
$$
This is a holomorphic system of ODE in the unknown $\Gamma$ that   has a solution in some neighborhood $W$ (see \cite{Hi97}). Now part (a) with local solvability yields that  starting with a local solution of $\omega_H(\Gamma')=u$ in a neighborhood of the point $a$, this solution can be analytically continued along any  continuous curve in $D$. Since $D$ is simply connected, the monodromy theorem for holomorphic maps  guarantees a global solution  $\Gamma:D\ra H$.
\end{proof}
Proposition~\ref{P:CEA} and Proposition~\ref{P:CMC} imply.
\begin{cor}\label{C:Ccomplete} 
\
\begin{enumerate}
    \item[a)] Suppose $0\in D\subset\bC$ is a simply connected domain and $u_0\in \mathfrak h^{1,0}$. There exists a holomorphic $\mathfrak D$-geodesic
$\Gamma:D\ra H$ with $\Gamma'(0)=u_0$ iff the Euler-Arnold vector field $\mathcal E$ has a holomorphic trajectory $u:D\ra \mathfrak h^{1,0}$ with initial condition $u(0)=u_0$.   
\item[b)]
 $\mathfrak D$ is $\bC$ complete iff $\mathcal E$  is $\bC$-complete.  
 \end{enumerate}
\end{cor}
From this we get.
\begin{cor} Let 
$\beta_\bC:\mathfrak h^{1,0}\times \mathfrak h^{1,0}\ra\mathfrak h^{1,0}$ be a symmetric, complex bilinear map. Suppose
   the holomorphic $(1,0)$ vector field $\mathcal V$ on  
$\mathfrak h^{1,0}$,
corresponding to the  map $u\mapsto -\beta_\bC(u,u)$ is $\bC$ complete. Then the left invariant torsion free holomorphic Koszul connection on $T^{1,0}H$ defined by
\begin{equation}
    \mathfrak D_uv:=\frac{1}{2}[u,v]+\beta_\bC(u,v),\qquad u,v\in\mathfrak h^{1,0}
\end{equation}
is $\bC$ complete.
\end{cor}
Corollary~\ref{C:Ccomplete} motivates our interest in $\bC$-complete and the more general $i\bR$-complete vector fields (c.f. Corollary~\ref{c:ircomplete}) the calculations and constructions of such vector fields are collected in  Section~\ref{S:vfields}.
\section{ \texorpdfstring{$\mathcal P$}P-complete  connections}
\subsection{Lift of a left invariant connection}\label{Ss:liftconn}
 Let $G$ be a real Lie group and $\widetilde{G}$  its universal cover with covering map $\sigma:\widetilde{G}\ra G$. 
    Denote the Lie algebra of $G$ (resp. of $\widetilde G$) by  $\mathfrak g$ (resp. $\widetilde{\mathfrak{g}}$) and the isomorphism $\sigma_*|_{\widetilde{\mathfrak{g}}}:\widetilde{\mathfrak{g}}\rightarrow\mathfrak{g}$ by $\varphi$. 
    Let $\nabla$ be a left invariant Koszul connection on $TG$ with corresponding bilinear map $\alpha:\mathfrak g\times\mathfrak g\rightarrow \mathfrak g$. 
Define
     $\widetilde \alpha:\widetilde{\mathfrak{g}}\times\widetilde{\mathfrak{g}}\rightarrow \widetilde{\mathfrak{g}}$ by  $\widetilde \alpha(\widetilde x,\widetilde y):=\varphi^{-1}\alpha(\varphi(\widetilde x),\varphi(\widetilde y))$ and let
    $\widetilde{\nabla}$ be the corresponding left-invariant Koszul connection on $T\widetilde G$.
\begin{prop}\label{P:univcconn}
     A smooth curve $\widetilde \gamma:I\rightarrow \widetilde G$ is a $\widetilde\nabla$-geodesic iff $\gamma=\sigma\circ \widetilde  \gamma$ is a $\nabla$-geodesic.
    Let $\widetilde v\in T\widetilde G$  and $\widetilde \gamma_{\widetilde v}:I\ra \widetilde G$ be the $\widetilde \nabla$-geodesic with $\widetilde \gamma'_{\widetilde v}(0)=\widetilde v$, 
   $v=\sigma_*\widetilde v$ and $\gamma_v=\sigma\circ\widetilde\gamma_{\widetilde v}$. Then ($\Omega_v$ is from \eqref{E:ueps})
   \begin{equation}\label{E:omegalift}
   \Omega_v=\sigma_*\circ\Omega_{\widetilde v}.  
   \end{equation}
\end{prop}
\begin{proof}
    Let $\omega_G$ (resp. $\omega_{\widetilde G}$) be the Maurer-Cartan form  of $G$ (resp. $\widetilde G$). Then $\varphi\omega_{\widetilde G}=\omega_G\circ\sigma_*$, since $\sigma$ is a homomorphism.  Let $\widetilde x:=\omega_{\widetilde G}(\dot {\widetilde{\gamma}})$ and  $x:=\omega_G(\gamma)$. Then
    \begin{equation}\label{E:covgeod}
        \varphi(\widetilde x)=\varphi(\omega_{\widetilde G}(\dot{\widetilde{\gamma}}))=\omega_G(\sigma_*\dot{\widetilde{\gamma}})=\omega_G(\dot \gamma)=x.
    \end{equation}
   Since $\varphi$ is linear, \eqref{E:covgeod} yields $\varphi(\dot{ \widetilde x})=\dot x$ and then from Proposition~\ref{P:leftinvconn} we get
    \begin{equation}
        \begin{split}
            \widetilde\gamma\text{ is a }\widetilde\nabla \text{-geodesic}&\iff \dot{\widetilde{x}}=-\widetilde\alpha(\widetilde x, \widetilde x) 
            \iff \varphi(\dot{\widetilde{x}})=\varphi(-\widetilde\alpha(\widetilde x, \widetilde x))\\ &\iff
            \dot{x}=-\alpha(x,x) 
            \iff \gamma \text{ is a }\nabla\text{-geodesic}.
        \end{split}
    \end{equation}
\eqref{E:omegalift} follows from the definitions.    
\end{proof}
\subsection{Universal complexification of a real Lie group}\label{Ss:univcompl}
 Recall, that a   connected, real Lie group $G$ always admits a  universal complexification, i.e. a complex Lie group $G_\bC$ with a  continuous homomorphism  $\eta:G\ra G_\bC$ that has the following property: for   any other complex Lie group $S$ and  continuous homomorphism
$\phi:G\ra S$ there exists a unique continuous homomorphism $\psi:G_\bC \ra S$  with $\phi=\psi\circ\eta$. When $G$ is compact or solvable, $\eta$ is a monomorphism but in general the kernel of $\eta$ can be nontrivial, even positive dimensional. The problem with $G_\bC$ is that it is not very explicit. Nevertheless its standard construction is as follows (\cite{Ho66}). Let $\widetilde G$ be the universal cover of $G$ with covering homomorphism $\sigma:\widetilde G\ra G$ and denote by $\mathfrak h$ the complexified Lie algebra, i.e.
$\mathfrak h:=\widetilde{\mathfrak g}\oplus\widetilde{\mathfrak g}$ with complex structure
$J:\mathfrak h\ra\mathfrak h$, $x\oplus y\mapsto -y\oplus x$. Let
$H$ be the unique simply connected complex Lie group with Lie algebra $\mathfrak h$, $\beta:\mathfrak g\ra \mathfrak h$, $x\mapsto x\oplus 0$  the canonical monomorphism and $\iota:\widetilde G\ra H$ the unique homomorphism with $\left.\iota_*\right|_{\mathfrak g}=\beta$. 
The universal complexification of $\widetilde G$ is in fact $\iota:\widetilde G\ra H$.
Since $\iota_*=\beta$ is a monomorphism,  $\iota$ will be an immersion. 
Let $B\triangleleft H$ be the smallest closed, complex, normal subgroup that contains $\iota(\ker \sigma)$. Denote by $pr:H\ra G_\bC$ the projection map. This defines a holomorphic principal bundle over $H/B$. Since $\ker\sigma\le\ker pr\circ\iota$,  we get a well defined induced homomorphism
$\eta:G\ra H/B$. Then 
$G_\bC:=H/B $  with $\eta:G\ra G_\bC$ will be the universal complexification of $G$.   $\eta(G)$ is  a closed, maximally totally real subgroup in $G_\bC$. From the definition we get
\begin{equation}\label{E:etadcomm}   
\eta\circ\sigma=pr\circ\iota.
\end{equation}
\subsection{Complexification of a left invariant Koszul connection  }\label{Ss:groupcompl}
Let $G$ be a  connected real Lie group with Lie algebra 
$\mathfrak g$ and $\eta:G\ra G_\bC$  the universal complexification as in Subsection~\ref{Ss:univcompl}. Suppose that
$\eta_*:T_eG\ra T_eG_\bC$ is a monomorphism (this holds if $G$ is either simply connected, or compact, or solvable or semisimple). Then $\eta_*(\mathfrak g)$ is a maximally totally real Lie subalgebra in $\mathfrak h$ (the Lie algebra of $G_\bC$) and  $\eta$ is an immersion.
Let 
$ \nabla$ be a  left invariant Koszul connection  on $TG$ with corresponding bilinear map $ \alpha:\mathfrak{g}\times\mathfrak g\ra \mathfrak g$ 
 and
 $\alpha_{\mathbb C}:\mathfrak h\times\mathfrak h\ra\mathfrak h$  the complex bilinear extension of $\alpha$, i.e. $\alpha^\bC(\eta_* u ,\eta_* v)=\alpha(u,v)$, $u,v\in\mathfrak h$.
 Denote by 
$\alpha^{1,0}:\mathfrak h^{1,0}\times\mathfrak h^{1,0}\ra\mathfrak h^{1,0}$ the map $(u^{1,0},v^{1,0})\mapsto (\alpha_\bC(u,v))^{1,0}$, $u,v\in\mathfrak h$   and let
$\mathfrak D$ be the corresponding
  left invariant holomorphic Koszul connection  on $T^{1,0}G_\bC$.  The underlying  real connection $\widehat{\nabla}$ on $TG_\bC$  (cf. Subsection \ref{Ss:geodflow}) then corresponds to the map $\alpha_\bC$, considered as an $\bR$-bilinear map.
Then
 $(G_\bC,\mathfrak D,\eta)$ will be  a complexification of $(G,\nabla)$ in the sense of Definition \ref{D:complexification}. 
\subsection{Ac-polarizations of left invariant connections}\label{Ss:acpolleftinv}
\begin{theorem}\label{T:eta-polarization}
Let $G$ be a connected, real Lie group and $\nabla$ a left invariant Koszul connection on $TG$. Let $\eta:G\ra G_\bC$ be the universal complexification. 
Then 
$N_i(\eta,G_\bC)\subset TG$ is open,  the associated polar map (cf. \eqref{E:Ni})
$\psi_\eta:N_i(\eta,G_\bC)\ra G_\bC$ is real-analytic
 and
$P:=(\psi_\eta)^{-1}_*(T^{1,0}G_\bC)$ is a $G$-invariant ac-polarization on  $N_i(\eta,G_\bC)$. 
\end{theorem}
\begin{proof}
    Let $\widetilde{G}$ be the universal cover of $G$ 
    and $\mathfrak g$, $\widetilde{\mathfrak g}$,  $\sigma$,
    $\iota$, $H$, $pr$ be as in Subsection~\ref{Ss:univcompl}.
    Define the 
    left-invariant Koszul connection $\widetilde \nabla$ on $\widetilde G$ as in Proposition~\ref{P:univcconn}. 
    \begin{lemma}\label{l:claim} 
  $\sigma_*N_i(\iota,H)=N_i(\eta,G_{\bC})$ and the following diagram commutes.
    \begin{equation}\label{di:diagram}
        \begin{tikzcd}
	{N_i(\iota,H)} & {H} \\
	{N_i(\eta,G_{\mathbb C})} & {G_{\mathbb C}}
	\arrow["{\psi_\iota}", from=1-1, to=1-2]
	\arrow["{\sigma_*}", from=1-1, to=2-1]
	\arrow["pr", from=1-2, to=2-2]
	\arrow["{\psi_\eta}", from=2-1, to=2-2]
    \end{tikzcd}
    \end{equation} 
    \end{lemma}
   \begin{proof}[Proof of Lemma~\ref{l:claim}]
  Let $\widetilde v\in T\widetilde G$   and  $\widetilde \gamma_{\widetilde v}$ the $\widetilde \nabla$-geodesic with ${\widetilde \gamma'_{\widetilde v}}(0)=\widetilde v$.  
   It follows  from  Proposition~\ref{P:univcconn} that $\gamma_v:=\sigma\circ\widetilde\gamma_{\widetilde v}$ is the $\nabla$-geodesic with 
   $\gamma_v'(0)=v:=\sigma_*\widetilde v$. 
Then from \eqref{E:etadcomm}
    \begin{equation}\label{E:eta}    \eta\circ\gamma_v=\eta\circ\sigma\circ\widetilde\gamma_{\widetilde v}=pr\circ\iota\circ \widetilde\gamma_{\widetilde v}.
    \end{equation}
If $\widetilde v\in N_i(\iota,H)$,
  the map $\iota\circ\widetilde{\gamma_{\widetilde v}}$ extends holomorphically to a neighborhood of $[0,i]$, denote this extension by $(\iota\circ\widetilde{\gamma_{\widetilde v}})^{\bC}$.       
Since $pr:H\ra G_\bC$ is holomorphic,  $(\eta\circ\gamma_v)^\bC:=pr\circ(\iota\circ\widetilde \gamma_{\widetilde v})^\bC$ will be holomorphic  in the same neighborhood. 
In light of \eqref{E:eta}, this map is a holomorphic extension of $\eta\circ\gamma_v$.
 This proves that $v\in N_i(\eta,G_\bC)$, i.e. $\sigma_*(N_i(\iota,H))\subset N_i(\eta,G_{\bC}) $.  We also get that
 $$
 pr\circ \psi_{\iota}(\widetilde v)=pr\circ(\iota\circ\widetilde \gamma_{\widetilde v})^{\bC}(i)=
            (\eta\circ\gamma_v)^{\bC}(i)
            =\psi_{\eta}(v).
$$
 Now suppose that $v\in N_i(\eta,G_{\bC})\cap T_aG$ and $\gamma_v$ is the $\nabla$-geodesic with $\gamma_v'(0)=v$. Then $\eta\circ\gamma_v$ extends holomorphically to a neighborhood of $[0,i]$, denote this extension by $(\eta\circ\gamma_v)^{\bC}$. Let $b\in\sigma^{-1}(a)$. From Proposition~\ref{P:univcconn} there exists a unique  $\widetilde \nabla$-geodesic $\widetilde \gamma$, such that $\widetilde\gamma(0)=b$ and $\sigma\circ\widetilde\gamma=\gamma_v$.
Let $\widetilde v=\widetilde\gamma'(0)$. Then 
$\sigma_*\widetilde v=v$. 
    We need to show that $\widetilde v\in N_i(\iota,H)$.
Since $pr:H\rightarrow G_\bC$ is a holomorphic fiberbundle, there exists a holomorphic map $\chi$ from some neighborhood of $[0,i]$ to $H$ with $pr\circ\chi=(\eta\circ\gamma_v)^{\bC}$. Because of \eqref{E:eta}, we can choose $\chi$ so that  for $t$ real and $|t|$ small enough, $\chi(t)=\iota\circ\widetilde\gamma$, 
therefore $\widetilde v\in N_i(\iota,H)$. Hence $N_i(\eta,G_\bC)\subset\sigma_*(N_i(\iota,H))$ finishing the proof of the claim.
\end{proof}  
 Back to the proof of the Theorem:
 since $\widetilde G$ is simply connected, its universal complexification will be $\iota:\widetilde G\ra H$ from Subsection~\ref{Ss:univcompl}. Then 
from Subsection~\ref{Ss:groupcompl} we know that $(H,\mathfrak D,\iota)$ is a complexification of $(\widetilde{G},\widetilde{\nabla})$, 
hence from Theorem~\ref{T:polarization} 
we get that $N_i(\iota,H)\subset T\widetilde G$ is open, the polar map $\psi_\iota:N_i(\iota,H)\ra H$ is real-analytic and $\widetilde P:=
 (\psi_\iota)^{-1}_*(T^{1,0}H)\subset \bC T(N_i(\iota,H))$ is an ac-polarization.
 Since $\sigma$ is a covering map, $\sigma_*:T\widetilde G\ra TG$ is a covering as well, so  the claim  implies 
 $N_i(\eta,G_{\bC})$ is open in $TG$, and $\psi_\eta$
is real-analytic.
    Using the commutativity of the diagram \eqref{di:diagram}, the holomorphicity of $pr$ and the definition of $P$ and $\widetilde P$ we get
    \begin{equation}\label{eq:sigma-1 P}
        \begin{split}
            (\sigma_*)^{-1}_*P&=(\sigma_*)^{-1}_*(\psi_{\eta})_*^{-1}(T^{1,0}G_{\bC})=(\psi_{\eta}\circ\sigma_*)_*^{-1}(T^{1,0}G_{\bC}) \\
            &=(\psi_\iota)_*^{-1}(pr_*)^{-1}T^{1,0}G_{\bC}=(\psi_\iota)_*^{-1}(T^{1,0}H)=\widetilde P.
        \end{split}
    \end{equation}
     Since $\sigma_*$ is a covering map and $\widetilde P$ is a smooth involutive bundle, \eqref{eq:sigma-1 P} yields that $P$ is a genuine smooth involutive bundle of the right rank, i.e. $P$ is a polarization.

    Next we show that $P$ is in fact an ac-polarization.
    Let $v\in TG$ and choose  $\widetilde v\in T\widetilde G$ with $\sigma_*\widetilde v=v$. 
    Recall $\Omega_{\widetilde v}$ from Definition \eqref{F:Omegav}. 
    Since $\widetilde P$ is an ac-polarization,  $\psi_\iota\circ\Omega_{\widetilde v}$ is holomorphic. 
   Using the commutativity of the diagram \eqref{di:diagram}, the holomorphicity of $pr$ and $\psi_\iota\circ\Omega_{\widetilde v}$ ($\widetilde P$ is an ac-polarization) and the fact that
   $\sigma_*\circ\Omega_{\widetilde v}=\Omega_v$ (c.f. \eqref{E:omegalift})
   we get
     $\psi_\eta\circ \Omega_v=\psi_\eta\circ\sigma_*\circ\Omega_{\widetilde v}=pr\circ\psi_\iota\circ\Omega_{\widetilde v}$ is holomorphic. 
     This shows, $P$ is an ac-polarization.

    The only thing left to prove is the 
    $G$-invariance of $P$. From the definitions one sees easily, that $N_i(\eta,G_{\bC})$ is $G$-invariant w.r. to the natural action of $G$ on $TG$. Furthermore, $G$ also acts on $G_{\bC}$ via the map $\eta$ by biholomorphisms and $T^{1,0}G_{\bC}$ is $G$-invariant as well. First we  prove that $\psi_\eta$ is $G$ equivariant!
    Let $v\in N_i(\eta,G_{\bC})$. Then $\eta\circ\gamma_v$ extends holomorphically to a neighborhood of $[0,i]$ by $\psi_\eta(v)=(\eta\circ\gamma_v)^{\bC}(i)$. Let $g\in G$ and $(L_g)_*v=w$. Now $g\cdot \gamma_v$ is a $\nabla$-geodesic with initial velocity $w$. Using that $\eta$ is a homomorphism $\eta(g\cdot \gamma_v)=\eta(g)\eta\circ\gamma_v$. This extends holomorphically to a neighborhood $[0,i]$ therefore 
    \begin{equation}\label{eq:eta homom}
        \psi_\eta((L_g)_*v)=\eta(g)\psi_\eta(v).
    \end{equation}
    Let $u\in P$, then using the definition of $P$, the holomorphicity of $L_{\eta(g)}$ and equation \eqref{eq:eta homom} we get
    \begin{equation}
        \begin{split}
            u\in P&\iff (\psi_\eta)_*u\in T^{1,0}G_\bC \\
            &\iff(L_{\eta(g)})_*(\psi_\eta)_*u\in T^{1,0}G_\bC \\
            &\iff (\psi_\eta\circ(L_g)_*)_*u\in T^{1,0}G_\bC\\
            &\iff ((L_g)_*)_*u\in P.   
        \end{split}
    \end{equation}
    This shows that $P$ is indeed $G$-equivariant.
\end{proof}
\subsection{\texorpdfstring{$\mathcal P$}P-complete left invariant
 connections}\label{Ss:pcompleteleftinv}
 \begin{definition}\label{d:iRcompl}
 A smooth real vector field $\mathcal W$ on $\bR^n$ is called   {\it $\bC$-complete} (resp. 
      {\it $i\bR$-complete}) if every  trajectory extends  to
a holomorphic map  $\bC\ra\bC^n$ (resp.  
$U\ra\bC^n$, where $U$ is
some open neighborhood of $i\bR$, that can depend on the trajectory).
 \end{definition}
 When $\mathcal W$ is homogeneous, $i\bR$-completeness is equivalent to: every $\mathcal W$-trajectory extends holomorphically to some open neighborhood of $[0,i]$. The neighborhood can depend on the trajectory. 
 See more on these type of vector fields in Section~\ref{S:vfields}.
\begin{theorem}\label{T:nabla-i-complete} Let $G$ be a connected, real Lie group with Lie algebra $\mathfrak g$, $\nabla$ a left invariant Koszul connection  on $TG$ and $\eta:G\ra G_\bC$ the universal complexification. 
Suppose the Euler-Arnold vector field  $\mathcal W$    is $i\bR$-complete.
Then $N_i(\eta,G_\bC)=TG$ and with the polar map $\psi_\eta$, associated to $\eta$,
$P:=(\psi_\eta)^{-1}_*(T^{1,0}G_\bC)$ is a $G$-invariant ac-polarization on $TG$.
In particular
$\nabla$ is $\mathcal P$-complete.
\end{theorem}
\begin{proof} It is enough  to show that 
$N_i(\eta,G_\bC)=TG$, the rest follows from 
Theorem \ref{T:eta-polarization}. 
Let $\sigma:\widetilde G\ra G$ be the universal 
cover, $\iota:\widetilde G\ra H$ the universal complexification as in Subsection~\ref{Ss:univcompl}, $\widetilde \nabla$ the left invariant Koszul connection on $T\widetilde G$ that is the lift of $\nabla$, as in Subsection~\ref{Ss:liftconn} and $\mathfrak D$ the left invariant holomorphic Koszul connection on
$T^{1,0}H$ and underlying real connection $\widehat\nabla$ on $TH$ as in Subsection~\ref{Ss:groupcompl}.
So $(H,\mathfrak D,\iota)$ is a complexification of $(\widetilde G,\widetilde \nabla)$.
Lemma~\ref{l:claim} shows that it suffices to prove $N_i(\iota,H)=T\widetilde G$.
The left $\widetilde G$-invariance of $\widetilde \nabla$ implies  $N_i(\iota,H)$ is left $\widetilde G$-invariant as well. So  it is enough to prove  $\widetilde{\mathfrak g}\subset N_i(\iota,H)$.
Let $\widetilde v\in \widetilde{\mathfrak g}$ and
   $\widetilde \gamma_{\widetilde v}$ be the $\widetilde{\nabla}$-geodesic with $\dot{\widetilde\gamma}_{\widetilde v}(0)=\widetilde v$. Then $\chi_v:=\iota\circ\gamma_v$ will be a $\widehat\nabla$-geodesic and so it is a restriction of a  holomorphic  $\mathfrak D$-geodesic $\Gamma_{\iota_*{\widetilde v}}$.
   Since $\mathcal W$ is $i\bR$-complete, the integral curve $c$, $c(0)=\widetilde v$ of $\mathcal W$ extends holomorphically (denoted by $c_\bC$) to a neighborhood of $[0,i]$.
   By the identity theorem, $c_\bC$ will be a holomorphic trajectory of the holomorphic Euler-Arnold vector field $\mathcal E$ of $\mathfrak D$ with initial condition $c_\bC(0)=\widetilde v^{1,0}$. Then
  Corollary~\ref{C:Ccomplete}(a) guarantees that the holomorphic $\mathfrak D$-geodesic 
  $\Gamma_{\iota_*{\widetilde v}}$ is defined on the same neighborhood, i.e. $\widetilde{\mathfrak g}\subset N_i(\iota,H)$
\end{proof}
A left invariant Koszul connection
on a disconnected Lie group $G$,  whose Euler-Arnold vector field is $i\bR$-complete will be also $\mathcal P$-complete. This is so, since every component of $G$ will be  connection preserving diffeomorphic to the unit component $G_0$  via left translations and we can apply 
Theorem~\ref{T:nabla-i-complete} to $\nabla$ restricted to $TG_0$.
Also from Theorem~\ref{T:nabla-i-complete} we obtain:
\begin{cor}\label{c:ircomplete} 
   Let $G$ be a connected, real Lie group with Lie algebra $\mathfrak g$ and $\beta:\mathfrak g\times\mathfrak g\ra\mathfrak g$ a symmetric bilinear map so that the vector field  $-\beta(x,x)$ is $i\bR$-complete. Then the left invariant connection
   $$
   \nabla_xy:=\frac1{2}[x,y]+\beta(x,y),\quad x,y\in\mathfrak g
   $$
   is $\mathcal P$-complete.
\end{cor}
The following theorem was proved in \cite{Sz05}, Theorem 0.2, (1) and (2). 
\begin{theorem}\label{T:locsymm}
Let $(M,\nabla)$ be an affine  locally symmetric Koszul manifold. Then $\nabla$ is $\mathcal P$-complete.
\end{theorem}
Let $G$ be a real Lie group and $\overset{c}\nabla$ its canonical connection.
Recall the  well known, formula of the curvature of $\overset{c}\nabla$ and its covariant derivative 
$$
R(x,y)z=-\frac1{4}[[x,y],z],\qquad x,y,z\in \mathfrak g,
$$
$$
\nabla R(x,y,z,w)=\frac1{8}[[x,y],[w,z]],\qquad w,x,y,z\in\mathfrak q.
$$
From these we obtain.
\begin{prop}\label{P:Cartanlocsymm}
$\overset{c}\nabla$ is flat iff $\mathfrak g$ is Abelian or 2-step nilpotent.
$\overset{c}\nabla$ is locally symmetric iff 
 $\mathfrak g$ is 2-step solvable.   
\end{prop}
Combining this Proposition with Theorem~\ref{T:locsymm} we can conclude that for a 2-step solvable Lie group $\overset{c}\nabla$ is $\mathcal P$-complete. In fact this latter property is always true 
\begin{cor}\label{C:cc}
Let $G$ be an arbitrary real Lie group. 
  Then $\overset{c}{\nabla}$ is $\mathcal P$-complete.
\end{cor}
\begin{proof}
 The Euler-Arnold vector field of the canonical connection is the identically zero vector field, that is obviously $i\bR$-complete, the statement then follows from Theorem \ref{T:nabla-i-complete}.   
\end{proof}
\begin{remark} 
 Unlike $\mathcal P$-completeness, 
 $\overset{c}\nabla$ is not  always entire. In fact
 it was proved in (\cite{Sz04}, Theorem 0.5), that this holds iff the Lie algebra  of $G$ is of Loeb type (i.e. for every $b\in\mathfrak g$, the only real number in the spectrum of $ad_b:\mathfrak g\ra\mathfrak g$ is zero). In light of  
 Corollary~\ref{C:cc} a real Lie group that is  not of Loeb type (for example  a noncompact semisimple Lie group  cf. \cite{Sz04}, Corollary 7.2) is an example  when  the canonical connection is  $\mathcal P$-complete but its ac-structure is not entire. 
\end{remark}
The following result (using our notations) 
was proved in \cite{ABI18}.
\begin{theorem}[Aslam-Burns-Irvine]\label{T:ABI}
Let $G$ be a compact, connected, real Lie group equipped with an entire Riemannian metric $g$. Then $(G_\bC,\mathfrak D,\eta)$ is a $\bC$-complete complexification and the polar map
$\psi_\eta:TG\ra G_\bC$ is biholomorphic.    
\end{theorem}
For a non-positive definite metric, it is not clear whether being entire is a stronger condition than $\bC$-completeness and if it is, whether $\psi_\eta$ is biholomorphic or not. Nevertheless it follows from both properties  that the metric is $\mathcal P$-complete.
\subsection{The group SU(2)}  
    Let $g$ be a left-invariant pseudo-Riemannian metric on $SU(2)$ and denote by  $\langle.,.\rangle$ the biinvariant Riemannian metric with constant sectional curvature $1$ (it is constant times the Killing form). Let  $A:su(2)\ra su(2)$ be the  $\langle.,.\rangle$-self adjoint linear transformation  such that $g(x,y)=\langle Ax,y\rangle$, $x,y\in su(2)$.
    Choose an orthonormal basis $\xi_1,\xi_2,\xi_3\in su(2)$
     consisting of eigenvectors of $A$ with $[\xi_i,\xi_{i+1}]=\xi_{i+2}$ mod $3$. Let $A\xi_j=\lambda_j\xi_j$. By renaming the indexes if necessary, we can assume that $\lambda_1\le\lambda_2\leq\lambda_3$. We say that $g$ is generic if $\lambda_1<\lambda_2<\lambda_3$ and call $g$ special if one of the inequality is an equality.
\begin{theorem}\label{T:SU(2)}
\
 \begin{enumerate}
     \item[a)]  If $g$ is generic,
     then it is not $\bC$-complete.
 \item[b)]    If $g$ is special, then it is $\mathcal P$-complete.
 \end{enumerate}   
\end{theorem}
\begin{proof} a)
    When $g$ is Riemannian, this was proved in \cite{ABI18}. But the same proof works literally for the pseudo-Riemann case as well.

    b) Since the Levi-Civita connection doesn't change if we multiply the metric with a nonzero scalar, we can assume that $\lambda_1=\lambda_2=1$ and $\lambda_3=\lambda$. A trajectory $x(t)=x_1(t)\xi_1+x_2(t)\xi_2+x_3(t)\xi_3$ of the Euler-Arnold vector field must satisfy the following system of ODE's:
     \begin{equation}\label{eq:EA SU2}
        \begin{split}
            \dot x_1&=(1-\lambda)x_2x_3 \\
            \dot x_2&=(\lambda-1)x_1x_3 \\
            \dot x_3&=0.
        \end{split}
    \end{equation}
    This was calculated in \cite{ABI18} in the Riemann case, but again their proof works in the pseudo-Riemann case as well.  Clearly the solutions of this system are defined and holomorphic on $\bC$,
    hence by Theorem \ref{T:nabla-i-complete}, the metric $g$ is $\mathcal P$-complete.
\end{proof}
\begin{remarks}\label{r:SU(2)}
    When $g$ is Riemannian and generic, 
    Ashlam-Burns-Irvine could conclude, using Theorem~\ref{T:ABI} and Theorem~\ref{T:SU(2)} (a), that
    $g$ cannot be entire. In the pseudo-Riemann case Theorem~\ref{T:ABI} is not available so we cannot say that a generic $g$ is not entire. It is possible to write down the Euler-Arnold system of ODE's in terms of $\lambda_1$, $\lambda_2$, $\lambda_3$ in the generic case as well, it will be the same  as in the Riemannian case obtained in \cite{ABI18}. Although we know that it is not a $\bC$-complete system, it may still be true that it is  
    $i\bR$-complete, from what we could conclude that $g$ is $\mathcal P$-complete. Unfortunately the solutions of this system have Jacobi elliptic functions as coordinate functions and it seems very difficult to locate their poles and consequently to see whether the system is $i\bR$-complete or not. 

    When $g$ is special, we can assume it is  normalized as $\lambda_1=\lambda_2=1$, $\lambda_3=\lambda$. For $0<\lambda$,  $g$ is Riemannian and it is  entire iff $\lambda\le1$ \cite{ABI18}. It is not known which special metrics with $\lambda<0$ will be entire.
\end{remarks}
\section{Biinvariant connections}
The Koszul connection $\nabla$ is called biinvariant if it is both left and right invariant on $TG$, where $G$ is a real Lie group.
The following characterization of biinvariant connections is well known, see for example \cite{PA13} for its proof.
\begin{theorem}
    The left-invariant torsion free Koszul connection
    $\nabla$ is biinvariant iff
    \begin{equation}\label{E:biinv-con}      [x,\nabla_yz]=\nabla_{[x,y]}z+\nabla_y[x,z],
    \end{equation}
    for all $x,y,z\in\mathfrak{g}$.
\end{theorem}
Since the Levi-Civita connection of a biinvariant pseudo-Riemannian metric is the canonical connection, we cannot obtain interesting biinvariant connections this way.  Also a result of Laquer \cite{Laq92} says that on a compact simple Lie group, except $SU(n)$ where $n\ge3$, the only biinvariant torsion free connection is the canonical connection. Nevertheless, there are interesting nontrivial biinvariant connections on noncompact Lie groups as we demonstrate later in Subsection~\ref{Ss:biinvFtype} and Subsection~\ref{Ss:biinv2nilp}. 
\section{ Lie groups  of type \texorpdfstring{$\mathcal{F}$}F}\label{S:Ftype}
A non-commutative Lie algebra $\mathfrak{g}$ is said to be of type  $\mathcal{F}$ if  for any $x,y\in\mathfrak g$,    $[x,y]$
is a linear combination of $x$ and $y$. Milnor showed \cite{M76} that this is equivalent to the existence of a linear functional $0\neq l\in\mathfrak{g}^*$ such that
\begin{equation}
    [x,y]=l(y)x-l(x)y.
\end{equation}
This implies that in every dimension up to isomorphism, there is only one such Lie algebra.
A non-commutative Lie group is said to be of type $\mathcal F$, if its Lie algebra is of type $\mathcal F$.
 Let $\mathfrak g$ be an $\mathcal F$ type Lie algebra and $G$ the unique simply connected Lie group with Lie algebra $\mathfrak g$. Consider the following realization of $G$ as a subgroup of the affine transformations of $\mathbb{R}^n$:
    $$\mathcal A^+(n,\mathbb R)=\left\{\left[\begin{matrix}
 aI_{n} & b\\0& 1
 \end{matrix}\right]\mid a\in\mathbb R, b\in\mathbb R^n , 0<a\right\}.
$$
$\mathcal{A}^+(n,\mathbb R)$ has dimension $n+1$, call its Lie algebra $\mathfrak{a}^+$. The smallest dimensional example is the group $\mathcal A^+(1,\mathbb R)$, the orientation-preserving affine parametrizations of the real line. Let us denote by $$e_0=\left[\begin{matrix}
 I_{n} & 0\\0& 0
 \end{matrix}\right],\quad\text{and}\quad e_j=e_{j,n+1},$$
where $e_{k,l}$ is an $(n+1)\times(n+1)$ matrix with the only nonzero entry being one in the k-th row and l-th column for $k,l=1,\dots,n$, and $span\{e_1,\dots,e_n\}=V$. Then $V$ is an abelian subalgebra of $\mathfrak{a}^+$, the derived subalgebra of $\mathfrak{a}^+$. This yields the following decomposition: $\mathfrak{a}^+=\langle e_0\rangle\oplus V$. Then any $x\in\mathfrak{a}^+$ can be written as $x=x_0e_0+x_v$, where $x_0\in\bR$ and $x_v\in V$. Hence the Lie bracket takes the following form 
\begin{equation}\label{E:F-Lie}
    [x,y]=x_0y_v-y_0x_v=x_0y-y_0y,\quad x,y\in\mathfrak{a}^+.
\end{equation}
 We also have that $ad_{e_0}|_V=Id_V$. 
\subsection{Biinvariant connections on 
\texorpdfstring{$\mathcal F$}F-type Lie groups}\label{Ss:biinvFtype}
\begin{theorem}\label{T:F-biinv-conn} Let $\nabla$ be a left-invariant torsion-free Koszul connection on $T\mathcal{A}^+(n,\bR)$. Then $\nabla$ is biinvariant if and only if there exists an $a\in\bR$ such that
\begin{equation}\label{E:biinvconn}
    \nabla_xy=\frac{1}{2}[x,y]+a(x_0y_0e_0+\frac{1}{2}x_0y_v+\frac{1}{2}y_0x_v)=\frac{a+1}{2}x_0y+\frac{a-1}{2}y_0x
\end{equation}
for any $x,y\in\mathfrak{a}^+$. Taking an arbitrary $a\in\bR$, the formula defines a $\mathcal P$-complete biinvariant connection.
\end{theorem}
The fact that formula \eqref{E:biinvconn} yields all biinvariant connections in the 2-dimesional case was observed in \cite{P19}.
\begin{proof}
    Suppose that $\nabla$ is biinvariant. Using \eqref{E:biinv-con} and \eqref{E:F-Lie}, with $x=e_0$ and $y,z\in\mathfrak{a}^+$ we get
    \begin{equation}\label{E:F-e_0}
    (\nabla_yz)_v=\nabla_{y_{v}}z+\nabla_y z_v.
    \end{equation}
    Using again \eqref{E:biinv-con} and \eqref{E:F-Lie}, with $x=x_v\in V$ and $y,z\in\mathfrak{a}^+$ we get
    \begin{equation}\label{E:F-V}
        -\left(\nabla_yz\right)_0 x_v=-y_0\nabla_{x_v}z-z_0\nabla_y x_v.
    \end{equation}
    If we substitute $y=z=x$ into \eqref{E:F-e_0} and $\eqref{E:F-V}$, we get
    \begin{align}
        (\nabla_x x)_v&=\nabla_{x_v}x+\nabla_x x_v,\label{eq:19}\\  
\left(\nabla_xx\right)_0x_v&=
x_0\nabla_{x_v}x+x_0\nabla_x x_v=x_0(\nabla_xx)_v.\label{eq:20}
    \end{align}
    Choosing $x_0=0$ in formula \eqref{eq:20} implies 
    $$\left(\nabla_{x_v}x_v\right)_0=0.$$
    Hence, substituting $x=x_v\in V$ into formula \eqref{eq:19} we obtain
    $$\nabla_{x_v} x_v=(\nabla_{x_v} x_v)_v=\nabla_{x_v}x_v+\nabla_{x_v} x_v=2\nabla_{x_v}x_v,$$
    therefore $\nabla_{x_v}x_v=0$.
    On the other hand substituting $x=e_0$ into formula \eqref{eq:20} we see that  $$(\nabla_{e_0}e_0)_0\cdot0=1\cdot(\nabla_{e_0}e_0)_v,$$
    which means that there exists an $a\in\bR$ such that $\nabla_{e_0}e_0=ae_0$. 
    From equation $\eqref{eq:19}$ 
    \begin{equation}
        \nabla_xx-\left(\nabla_xx\right)_0e_0=(\nabla_xx)_v=\nabla_{x_v}x+\nabla_{x}x_v,
    \end{equation}
    thus
    \begin{equation}    x_0^2\nabla_{e_0}e_0+x_0\nabla_{x_v}e_0+x_0\nabla_{e_0}x_v+\nabla_{x_v}x_v-\left(\nabla_xx\right)_0e_0=2\nabla_{x_v}x_v+x_0\nabla_{x_v}e_0+x_0\nabla_{e_0}x_v.
    \end{equation}       
    Using $\nabla_{x_v}x_v=0$ we can conclude 
    \begin{equation}\label{eq:21}
        x_0^2\nabla_{e_0}e_0-\left(\nabla_xx\right)_0e_0=x_0^2ae_0-\left(\nabla_xx\right)_0e_0=0.
    \end{equation}
    By equation \eqref{eq:20} and \eqref{eq:21}
    \begin{equation}
        x_0\nabla_x x=x_0((\nabla_xx)_v+(\nabla_xx)_0e_0)=(\nabla_xx)_0x_v+x_0(\nabla_xx)_0e_0=ax_0^2(x_v+x_0e_0).
    \end{equation}
    Therefore
    \begin{equation}\label{eq:22}
        \nabla_xx=ax_0x.
    \end{equation}
    Since  $\nabla$ is torsion free,  the polarization identity for the symmetric bilinear form
    $\beta:\mathfrak g\times\mathfrak g\ra\mathfrak g$, $\beta(x)=\nabla_xx$ together with \eqref{eq:22} yields 
   \eqref{E:biinvconn}.
    Suppose now, that $\nabla$ is defined by the formula \eqref{E:biinvconn}. Then
\eqref{E:biinvconn} and \eqref{E:F-Lie} yield 
    \begin{align*}
        [x,\nabla_yz]&=\frac{1}{2}[x,[y,z]]+a\left(x_0\frac{1}{2}(y_0z_v+z_0y_v)-y_0z_0x_v\right), \\
        \nabla_{[x,y]}z&=\frac{1}{2}[[x,y],z]+a\left(\frac{1}{2}z_0(x_0y_v-y_0x_v)\right), \\
        \nabla_y[x,z]&=\frac{1}{2}[y,[x,z]]+a\left(\frac{1}{2}y_0(x_0z_v-z_0x_v)\right).
    \end{align*}    
    Adding the last two formulas gives the first which in light of \eqref{E:biinv-con} proves the biinvariance of $\nabla$.
     From formula~\eqref{E:biinvconn} we also get  $\nabla_xx=a x_0x$. Therefore the integral curves of the Euler-Arnold vector field of $\nabla$ are the solutions of
    \begin{equation}
        \dot{x}=-ax_0x.
    \end{equation}
    Proposition~\ref{P:l-L-Omega} a) tells us that such a vector field is  $i\bR$-complete, hence by Theorem~\ref{T:nabla-i-complete}   $\nabla$ is $\mathcal P$-complete.
\end{proof}
\begin{prop}\label{P:F-biinvariant-curvature}
    Let  $a\in\bR$ and $\nabla$ the  Koszul connection on $T\mathcal{A}^+(n,\bR)$ defined by \eqref{E:biinvconn}. Then for $x,y,z,w\in\mathfrak{a}^+$
    \begin{equation}\label{E:curvature-F-biinvariant}
        R(x,y)z=\frac{a^2-1}{4}z_0(y_0x-x_0y),
    \end{equation}
     and
    \begin{equation}\label{E:covariant-F-biinvariant}
        (\nabla_wR)(x,y,z)=\frac{a^2-1}{2}aw_0z_0(x_0y-y_0x).
    \end{equation}   
\end{prop}
\begin{proof}
    First, since $\nabla$ is biinvariant and torsion-free, equation \eqref{E:biinv-con} yields
    \begin{equation}\label{E:biinv-curvature}
        \begin{split}
        R(x,y)z&=\nabla_x\nabla_yz-\nabla_y\nabla_xz-\nabla_{[x,y]}z \\
        &=\nabla_x\nabla_yz-\nabla_y\nabla_xz-[x,\nabla_yz]+\nabla_y[x,z]\\
        &=\nabla_x\nabla_yz-\nabla_y\nabla_xz-(\nabla_x\nabla_yz-\nabla_{\nabla_yz}x)+\nabla_y(\nabla_xz-\nabla_zx)\\
        &=\nabla_{\nabla_yz}x-\nabla_y\nabla_zx.
    \end{split}
    \end{equation}
    Using  \eqref{E:biinv-curvature} and \eqref{E:biinvconn} we can calculate the curvature of $\nabla$ as follows:
    \begin{equation}
        \nabla_{\nabla_yz}x=\frac{a+1}{2}ay_0z_0x+\frac{a-1}{2}x_0\left(\frac{a+1}{2}y_0z+\frac{a-1}{2}z_0y\right),
    \end{equation}
    and
    \begin{equation}
        \nabla_y\nabla_zx=\frac{a+1}{2}y_0\left(\frac{a+1}{2}z_0x+\frac{a-1}{2}x_0z\right)+\frac{a-1}{2}az_0x_0y,
    \end{equation}
    which yields  \eqref{E:curvature-F-biinvariant}. Now  $(\nabla_wR)(x,y,z)$ is defined by 
    \begin{equation}\label{E:covariant curvature}
        (\nabla_wR)(x,y,z)=\nabla_w(R(x,y)z)-R(\nabla_wx,y)z-R(x,\nabla_wy)z-R(x,y)\nabla_wz.
    \end{equation}
    Using  \eqref{E:biinvconn} and \eqref{E:curvature-F-biinvariant} we can calculate the terms on the right hand side one by one.
    \begin{equation}\label{eq:43}
        \nabla_w(R(x,y)z)=\frac{a+1}{2}w_0\frac{a^2-1}{4}z_0(y_0x-x_0y),
    \end{equation}
    \begin{equation}\label{eq:44}
        R(\nabla_wx,y)z=\frac{a^2-1}{4}z_0\left(y_0\left(\frac{a+1}{2}w_0x+\frac{a-1}{2}x_0w\right)-aw_0x_0y\right),
    \end{equation}
    switching the roles of $x$ and $y$ and using the identity $R(u,v)=-R(v,u)$ for $u,v\in\mathfrak{a}^+$ we obtain 
    \begin{equation}\label{eq:45}
        R(x,\nabla_wy)z=-\frac{a^2-1}{4}z_0\left(x_0\left(\frac{a+1}{2}w_0y+\frac{a-1}{2}y_0w\right)-aw_0y_0x\right),
    \end{equation}
and finally
    \begin{equation}\label{eq:46}
        R(x,y)\nabla_w z=\frac{a^2-1}{4}aw_0z_0(y_0x-x_0y).
    \end{equation}
    Substituting  \eqref{eq:46}, \eqref{eq:45}, \eqref{eq:44} and \eqref{eq:43} into \eqref{E:covariant curvature} we obtain \eqref{E:covariant-F-biinvariant}.
\end{proof}
 From Theorem \ref{T:F-biinv-conn} and Proposition \ref{P:F-biinvariant-curvature} we get:
\begin{theorem}
    Let $a\in\bR$  and $\nabla$ the corresponding biinvariant Koszul connection on $T\mathcal{A}^+(n,\bR)$ defined by formula \eqref{E:biinvconn}. When $a=0$, $\nabla$  is the canonical connection, hence locally symmetric. When $a=\pm1$, $\nabla$ is flat, so locally symmetric.  However for $a\in\bR\setminus\{-1,0,1\}$ the connection is not locally symmetric but it is $\mathcal P$-complete.
\end{theorem}
\subsection{Left invariant pseudo-Riemannian metrics on \texorpdfstring{$\mathcal F$}F-type Lie groups}
 Let  $G$ be the unique simply connected 
 $\mathcal F$-type Lie group. Here we do not need the specific realization of $G$ used in Subsection~\ref{Ss:biinvFtype}.
 Milnor \cite{M76} proved that  any left invariant Riemannian metric $g$ on $TG$  has constant negative sectional curvatures. Since $g$ is automatically complete, $(G,g)$ will be isometric to the real hyperbolic space $\mathbb R H^n$ with
 $n=\dim\mathfrak g$. It is known (\cite{Lau03}), that $G$ admits essentially only one left invariant Riemannian metric (i.e. up to scaling and $Aut G$).
  Wolf proved  \cite{W64} that $G$ admits a left invariant flat Lorentz metric. This was extended by Nomizu \cite{N79}  showing that there also exists left invariant Lorentz metrics with constant sectional curvature $c$  for arbitrary  $c\in \bR$. He also showed that every left invariant Lorentz metric has constant sectional curvature. Finally  
   left invariant pseudo-Riemannian metrics of arbitrary signature were classified in \cite{KOTT16}, \cite{VS20}, they all have constant sectional curvatures. 
\begin{theorem}\label{T:type F P-complete} Let  $G$ be an $\mathcal F$-type Lie group and $g$ a left invariant pseudo-Riemannian metric on $TG$. Then  $g$ is $\mathcal P$-complete.  It is entire  iff it is flat.
\end{theorem}
\begin{proof} Let $(\widetilde G,\widetilde g)$ be the universal cover of $G$ with the induced left invariant metric $\widetilde g$.
From what was said above, the   sectional curvatures of $\widetilde g$, hence of $g$ as well, are constant, denote it by $c$.
 This implies that   $(G,g)$ is locally symmetric (\cite{O83}), hence by
 Theorem~\ref{T:locsymm}  $g$ is $\mathcal P$-complete.
  By \cite{Sz04} Theorem 8.4,  a pseudo-Riemannian space form with nondefinite metric is entire iff its sectional curvatures are zero. This    implies that a nondefinite left invariant pseudo-Riemannian metric with $c\not=0$ is not entire.
  When
  $g$ is definite, we can assume it is Riemannian (since $g$ and $-g$ has the same Levi-Civita connection) 
  and so $c$ is negative. But by Theorem 4.3 in \cite{LSz91}, all the sectional curvatures of an entire metric must be nonnegative. So $g$ is not entire.
   When $c=0$, $(G,g)$ is locally isometric to a pseudo-Euclidean space of appropriate signature,  so $g$ will be entire.  
\end{proof}
\begin{theorem}\label{T:type F i-complete}
    Let $g$ be a left-invariant pseudo-Riemannian metric 
    on a Lie group $G$ of type $\mathcal F$.
    Then the Euler-Arnold vector field is 
    $i\bR$-complete iff either $g$   is degenerate on $\mathfrak g'$ (the derived  Lie subalgebra),  or $g$  is definite on $\mathfrak g'$ but indefinite on $\mathfrak g$. 
\end{theorem}
\begin{proof} We take the realization 
$\mathcal{A}^+(n,\bR)$ of $G$ and use the notation from the beginning of Section~\ref{S:Ftype}.

Suppose that $g|_V$ is nondegenerate. Choose an orthogonal basis $f_1,\dots, f_n\in V$ and let 
$$f_0:=e_0-\sum_{j=1}^{n}\frac{g(e_0,f_j)}{g(f_j,f_j)}f_j.$$
Now, $ad_{f_0}|_V=ad_{e_0}|_V$, since $V$ is abelian. According to Corollary~\ref{C:EAequation}, a trajectory $x$, $x=x_0f_0+x_V$ of the  Euler-Arnold vector field  must satisfy the following system of ODE's
\begin{equation}\label{eq:nondeg-f}
    \dot{x}_0(t)=-\frac{g(x_V,x_V)}{g(f_0,f_0)}= x_0^2(t)-\frac{g(x,x)}{g(f_0,f_0)},\quad 
    \dot{x}_V(t)=x_0(t)x_V(t).
\end{equation}
From Proposition \ref{P:q-v} with the choice of $g$ and $v=\frac{f_0}{g(f_0,f_0)}$ we get that the Euler-Arnold vector field is $i\bR$-complete iff $g|_V$ is definite, but indefinite on $\mathfrak{a}^+$.

Suppose $g|_V$ is degenerate. Then it has rank $n-1$ with a single lightlike direction spanned by $f_1\in V$. We can assume that $g(f_1,e_0)=1$. Choose a direct summand $U$ to $\langle f_1\rangle$ in $V$ and an orthogonal basis $f_2, \dots,f_n$ in $U$. With these choices there is a unique vector  $f_0$ with the properties $g(f_0,f_1)=1$ (or equivalently $ad_{f_0}|_V=Id|_V$), $f_0\perp_g U$ and $g(f_0,f_0)=0$.  It is defined by the formula
$$f_0:=e_0+\left(\sum_{j=2}^{n}\frac{g(e_0,f_j)^2}{2g(f_j,f_j)}-\frac{g(e_0,e_0)}{2}\right)f_1-\sum_{j=2}^{n}\frac{g(e_0,f_j)}{g(f_j,f_j)}f_j.$$
For $x\in\mathfrak{a}^+$, let $x=x_0f_0+x_1f_1+x_U$.  According to Corollary~\ref{C:EAequation}
a trajectory of the Euler-Arnold vector field $\mathcal W$ is a solution of
\begin{align*}
    \dot{x}_0(t)&=x_0^2(t), \\
    \dot{x}_1(t)&=-x_0(t)x_1(t)-g(x_U(t),x_U(t))=x_0(t)x_1(t)-g(x(t),x(t)),\\
    \dot{x}_U(t)&=x_0(t)x_U(t).
\end{align*}
With the choice of $v=f_1$,  $\mathcal W$ thus takes the form required by 
Proposition \ref{P:q-v}. Since $f_1$ is lightlike,  the Proposition yields that $\mathcal W$  is indeed $i\bR$-complete.
\end{proof}
\begin{cor} For left-invariant pseudo-Riemannian metrics
    $i\bR$-completeness of the Euler-Arnold vector field is a strictly stronger condition  than $\mathcal{P}$-completeness. 
\end{cor}
\begin{proof} This follows from Theorem~\ref{T:type F P-complete} and Theorem ~\ref{T:type F i-complete}.
\end{proof}
\section{\texorpdfstring{$2$}2-step nilpotent Lie groups}
\subsection{Biinvariant connections}\label{Ss:biinv2nilp}
\begin{theorem}   
    Let $G$ be a real 2-step nilpotent Lie group and let $\mathfrak{g}=W\oplus V$, where $W$ is a linear subspace and $V$ is the center of the Lie algebra. Let $l\in W^*$ and $\Omega:W\times W\rightarrow V$ a symmetric bilinear map. The following formula defines a  biinvariant, torsion-free,
    $\mathcal P$-complete connection. 
    \begin{equation}\label{E:2-step biinv}
        \nabla_{x}y=\frac{1}{2}[x,y]+l(x_w)y+l(y_w)x+\Omega(x_w,y_w.
    \end{equation}
   $\nabla$ is flat when $l=0$, but it is not even locally symmetric otherwise.  
\end{theorem}
\begin{proof}
   Let  $x,y,z\in\mathfrak{g}$, then \eqref{E:2-step biinv} implies
    \begin{equation}\label{prf 2-step biinv1}
        [x,\nabla_{y}z]=l(y_w)[x,z]+l(z_w)[x,y]
    \end{equation}
    and 
    \begin{equation}\label{prf 2-step biinv2}
        \nabla_{[x,y]}z=l(z_w)[x,y],\qquad
            \nabla_{y}[x,z]=l(y_w)[x,z].
    \end{equation}
    Equations \eqref{prf 2-step biinv1} and \eqref{prf 2-step biinv2} show that formula \eqref{E:biinv-con} holds, proving the biinvariance of the connection. 
    $\nabla$ is  torsion free, since the antisymmetric part of $\alpha(x,y)=\nabla_xy$
    is $[x,y]/2$.
    Now by Proposition~\ref{P:leftinvconn} the integral curves of the  Euler-Arnold vector field $\mathcal W$ must satisfy $\dot x=-\nabla_xx$,  i.e. with $x=x_w+x_v$, the following system of equations 
    \begin{align}
        \dot{x}_w&=-2l(x_w)x_w\label{E:W} \\    
        \dot{x}_v&=-2l(x_w)x_v-\Omega(x_w,x_w)\label{E:V}
    \end{align}
    Proposition~\ref{P:l-L-Omega} (a) and (b) implies that  the system \eqref{E:W} and \eqref{E:V} i.e.  $\mathcal W$  is $i\bR$-complete. 
    Hence by Theorem \ref{T:nabla-i-complete}    the connection is $\mathcal P$-complete.
     Routine calculation, using the fact that $G$ is 2-step nilpotent together with formula \eqref{E:2-step biinv}
    gives the formula of the curvature tensor 
 \begin{equation}
    R(x,y)z=l(z_w)(l(y_w)x-l(x_w)y)
\end{equation}   
and its covariant derivative
\begin{equation}
    (\nabla_wR)(x,y,z)=4l(w_w)l(z_w)(l(x_w)y-l(y_w)x),
\end{equation}
finishing the proof.
\end{proof}
When the center is one dimensional, we can show that all biinvariant connections have the form 
\eqref{E:2-step biinv}. For this purpose we take the standard realizations of such groups.
\begin{definition}
    Consider the following subgroup of $GL(n+2,\bR)$
  $$  \mathcal H(n,\mathbb R)=\left\{\left[\begin{matrix}
 1 & a^T & c\\
 0 & I_{n} & b\\
 0 & 0 & 1
 \end{matrix}\right]\mid c\in\mathbb R, a,b\in\mathbb R^n \right\}.$$
$\mathcal{H}(n,\bR)$ is a 2-step nilpotent Lie group of dimension $2n+1$. Call its Lie algebra $\mathfrak{h}$. Let us denote by $W$ the span of $e_{1,k}$ and $e_{k,n+2}$ for $k=2,\dots n+1$, 
$v:=e_{1,n+2}$
and  $V:=\langle v\rangle$. Then $\mathfrak h=W\oplus V$ and $V$ is the center as well as the derived subalgebra of $\mathfrak{h}$. There is 
a symplectic structure
on $W$ defined by 
$$\omega=\sum_{k=2}^{n+1} e^*_{1,k}\wedge e^*_{k,n+2}.$$
Using $\omega$ for $x=x_w+x_vv$ and $y=y_w+y_vv$, where $x_w,y_w\in W$
\begin{equation}\label{E:H-Lie}
    [x,y]=\omega(x_w,y_w)v.
\end{equation}
\end{definition}
\begin{theorem}\label{T:Heisenberg-biinvariant}
    Let $\nabla$ be a biinvariant, torsion-free Koszul connection on $T\mathcal{H}(n,\bR)$. Then $\nabla$ is of the form
    \eqref{E:2-step biinv}. In particular it is
    $\mathcal P$-complete.
\end{theorem}
\begin{proof}
    Using \eqref{E:biinv-con} and \eqref{E:H-Lie} we obtain the identity
    \begin{equation}\label{eq:23}
        \omega(x_w,(\nabla_y z)_w)v=\omega(x_w,y_w)\nabla_vz+\omega(x_w,z_w)\nabla_yv,
    \end{equation}
    for $x,y,z\in\mathfrak{h}$. Substituting $y=z$  into \eqref{eq:23} we can see that 
    $$\omega(x_w,(\nabla_y y)_w)v=\omega(x_w,y_w)(\nabla_vy+\nabla_yv).$$
    Therefore there must exist an $L\in \mathfrak{h}^*$ such that $\nabla_vy+\nabla_yv=2L(y)v.$ Using the facts that $\nabla$ is torsion-free and $v$ is in the center, we obtain  $\nabla_yv-\nabla_vy=[y,v]=0,$. Hence
    \begin{equation}\label{E:H-l}
        \nabla_y v=\nabla_v y= L(y)v.
    \end{equation}
    Equation \eqref{E:H-l} and \eqref{eq:23}  then yields
    \begin{equation}\label{eq:24}
        \omega(x_w,(\nabla_y y)_w)v=\omega(x_w,y_w)2L(y)v=\omega(x_w,2L(y)y_w)v.
    \end{equation}
     By the nondegeneracy of $\omega$, this implies that
    \begin{equation}\label{E:H-l-w}
        (\nabla_y y)_w=2L(y)y_w.
    \end{equation}
    On the other hand, if we substitute $z=v$ into equation \eqref{eq:23} we obtain
    \begin{equation}
        \omega(x_w,(\nabla_y v)_w)v=\omega(x_w,y_w)\nabla_v v.
    \end{equation}
    By equation \eqref{E:H-l} the LHS here is zero, and $\omega$ is nondegenerate so $\nabla_vv=0$ as well. Combining this equality with equation \eqref{E:H-l} we get  
    \begin{equation}\label{E:H-l-v}
        L(v)v=\nabla_vv=0. 
    \end{equation}
    In summary, we obtain the following identity using equations \eqref{E:H-l-v}, \eqref{E:H-l-w} and \eqref{E:H-l}:
    \begin{align*}
        \nabla_xx&=\nabla_{x_w+x_vv}(x_w+x_vv)\\ &=\nabla_{x_w}x_w+x_v(\nabla_{x_w}v+\nabla_vx_w)+x_v^2\nabla_vv\\
        &=2L(x_w)x_w+(\nabla_{x_w}x_w)_vv+2L(x_w)x_vv\\
        &=2L(x_w)x+(\nabla_{x_w}x_w)_vv
    \end{align*}
    Since by our assumption $\nabla$ is torsion-free, this yields that the connection takes the form of \eqref{E:2-step biinv} with $l:=\left.L\right|_W$, and 
    $\Omega(u,z):=(\nabla_uz+\nabla_zu)_vv$, where $u,z\in W$.
\end{proof}
\subsection{Left invariant metrics}
\begin{theorem}\label{T:2step}
 Let $H$ be a complex 2-step nilpotent Lie group and $q$ a left in\-va\-ri\-ant holomorphic Riemannian metric on $T^{1,0}H$. Then $q$ is $\bC$-complete.   
\end{theorem}
The theorem fails already for 3-step nilpotent Lie groups. Guediri \cite{G94} exhibited an example   of an incomplete pseudo-Riemannian metric on a certain 3-step nilpotent Lie group. The complexification of that metric then will not be $\bC$-complete on the complexified group. Guediri also proved in the same paper 
the real version of the theorem i.e. that  every left invariant pseudo-Riemannian metric on a real, 2-step nilpotent Lie group is
$\bR$-complete.  
\begin{proof}
Let  $\mathfrak h$ be the Lie algebra of $H$. In light of Corollary~\ref{C:Ccomplete} a), we need to show that the 
\begin{equation}\label{E:EA}
\dot x=(ad_x)^*x
\end{equation}
Euler-Arnold ODE system is $\bC$ complete. 
Let $\mathfrak h'$ be the derived algebra and $V$  a vector subspace in $\mathfrak h$ so that $\mathfrak h=\mathfrak h'\oplus V$. Let $e_1,\dots,e_n$ be a $\bC$ basis in 
$\mathfrak h'$ and $e_{n+1},\dots, e_{n+k}$ a $\bC$ basis in $V$. For $x\in\mathfrak h$ denote by $A_x$ the matrix of $ad_x$  and by
$Q$  the nondegenerate symmetric matrix of $q$ on $\mathfrak h$ in this bases i.e. $Q=(q_{jk})$, where $q_{jk}=q(e_j,e_k)$.
Then $Q^{-1}A_x^TQ$ will be the matrix of $ad_x^*$.
Let
 $z:=Qx$, where now $x=\sum x_je_j\in\mathfrak h$ is identified with $(x_1,\dots,x_{n+k})\in\bC^{n+k}=\bC^n\times\bC^k$.  We get that the system
\eqref{E:EA} is equivalent to
\begin{equation}\label{E:EA2}
 \dot z=A^T_{Q^{-1}z}z, \qquad z\in\bC^{n+k}.   
\end{equation}
It is enough to show that this system is $\bC$ complete. 
  Since $\mathfrak h$ is 2-step nilpotent, for any $x\in \mathfrak h$,  and $j=1,\dots, n$ we have $[x,e_j]\in[\mathfrak h, \mathfrak h']=0$. Also $[x,e_{n+s}]\in\mathfrak h'$, $s=1,\dots,k$. Thus the matrix $A_x$ of $ad_x$  has the following  form 
$$
A_x=\left(\begin{matrix}
    0 & \alpha_x\\ 0&0
\end{matrix}\right),
$$
where $\alpha_x$ is an $n\times k$ matrix whose elements depend linearly on $x$.
Thus 
\begin{equation}\label{E:transpose}
A^T_x=\left(\begin{matrix}
    0 & 0\\ \alpha^T_x&0
\end{matrix}\right).
\end{equation}
Let $\varphi(z,w)=\frac1{2}(A^T_{Q^{-1}z}w+A^T_{Q^{-1}w}z)$, $z,w\in\bC^{n+k}$. Then $\varphi$ is symmetric and bilinear.  \eqref{E:transpose} yields that  $\varphi:\bC^{n+k} \times\bC^{n+k}\ra 0\times\bC^k$  
and $\varphi(z,w)=0$ if $z, w\in 0\times\bC^k$. Now Proposition~\ref{P:complete} shows that the vector field corresponding to $P(z)=\varphi(z,z)$ is $\bC$-complete,
 yielding that the system $\eqref{E:EA2}$ is $\bC$ complete.
\end{proof}
\begin{theorem}\label{T:2stepcomplete}
     Let $G$ be  a real 2-step nilpotent Lie group and $g$ a  left invariant pseudo-Riemann metric. Then $g$  is $\mathcal P$-complete.    
\end{theorem}
\begin{proof} Let  $\mathfrak h$ be the complexification of $\mathfrak g$ and denote by $q:\mathfrak h\times\mathfrak h\ra\bC$ the complex bilinear extension of $\left.g\right|_\mathfrak g$. The proof of Theorem~\ref{T:2step} shows that the system
$$
\dot x=ad^*_xx
$$
is $\bC$-complete on $\mathfrak h$, where now
$ad^*_x$ is defined w.r.t. $q$. In particular the Euler-Arnold vector field of $g$ is $\bC$-complete. 
   Then  $\mathcal P$-completeness follows from Theorem~\ref{T:nabla-i-complete}.   
\end{proof}
Since every generalized Heisenberg Lie algebra is 2-step nilpotent, Theorem~\ref{T:2stepcomplete} implies that the generalized Heisenberg groups with the left invariant metrics studied in \cite{HI03} are $\mathcal P$-complete.
\subsection{The 3-dimensional Heisenberg group}
Let $H_3$ be the 3-dimensional Heisenberg group.
It is well known, that up to isometry there is only one left invariant Riemannian metric on $H_3$. This metric was investigated  by Halverscheid and Iannuzzi in
 \cite{HI03}. They showed  that this metric is not entire, but there is a maximal domain $\Omega\subset TH_3$, where the ac-structure can be defined. $\Omega$ has  a complicated real analytic boundary, it is neither holomorphically separable nor holomorphically convex, its envelope of holomorphy is biholomorphic to $\bC^3$.   
 The left invariant Lorentz metrics on $H_3$ were classified in \cite{R92}, \cite{RR06}. Up to isometry there are three kind. Two come as a 1-parameter family, the third is flat, hence its
 metric is entire. It is not known, what are the maximal domain of definition (if they exists)
 of the first two families of Lorentz metrics.
  Nevertheless Theorem \ref{T:2stepcomplete} yields, that all  left invariant metrics on $H_3$ are $\mathcal P$-complete.
\section{Lie groups whose Lie algebra admits a reductive decomposition}
\subsection{Deformations of the canonical connection}
\begin{definition} Let $\mathfrak g$ be a Lie algebra over the field $\mathbb K$
(where $\mathbb K=\bR$ or $\mathbb K=\bC$) and $\mathfrak k\subset\mathfrak g$ a Lie subalgebra over $\mathbb K$. We shall say that $(\mathfrak g,\mathfrak k)$ is a {\it reductive pair} over $\mathbb K$ if there exists a $\mathbb K$-linear subspace $\mathfrak p$ in $\mathfrak g$ with $\mathfrak g=\mathfrak k\oplus \mathfrak p$ and $[\mathfrak k, \mathfrak p]\subset\mathfrak p$. We call $\mathfrak g=\mathfrak k\oplus \mathfrak p$  a {\it reductive decomposition} of $\mathfrak g$.
\begin{examples} Let $\mathfrak g$ be a real Lie algebra and $\theta:\mathfrak g\ra \mathfrak g$ an involutive automorphism. Then $\mathfrak k:=\{x\in\mathfrak g \mid \theta x=x\}$, $\mathfrak p:=\{x\in \mathfrak g \mid \theta x=-x\}$ yield a reductive decomposition of $\mathfrak g$, in which case $[\mathfrak p,\mathfrak p]\subset\mathfrak k$ also holds. As a special case let $\mathfrak g$ be  an arbitrary real Lie algebra and  $\mathfrak h:=\mathfrak g\otimes\bC$ considered as a real Lie algebra with involution
$\theta:\mathfrak h\ra\mathfrak h$, $\theta:u\mapsto \overline u$. Then $\mathfrak h=\mathfrak g\otimes 1\oplus \mathfrak g\otimes i$ is a reductive decomposition.   
\end{examples}
Let $G$ be a real  Lie group with Lie algebra $\mathfrak g$ and reductive decomposition   $\mathfrak k\oplus\mathfrak p=\mathfrak g$. Let   $\delta$ be a real number.  
The  {\it $\delta$-deformation} of the canonical connection is
   the left invariant torsion free Koszul connection $\nabla^\delta$   on $TG$ 
defined by
\begin{equation}\label{E:conn}
\nabla^\delta_xy
=\frac1{2}\left([x,y]+\delta\left([x_\mathfrak p,y_\mathfrak k]+[y_\mathfrak p,x_\mathfrak k]\right)\right),
\end{equation}
where  $x=x_\mathfrak k+x_\mathfrak p$, $y=y_\mathfrak k+y_\mathfrak p$ are the decompositions in $\mathfrak k\oplus\mathfrak p=\mathfrak g$.
   
 Similarly let $H$ be a complex Lie group with Lie algebra $\mathfrak h$ and reductive decomposition   $\mathfrak l\oplus\mathfrak m=\mathfrak h$. Let  $\zeta\in \bC$.
 The {\it $\zeta$-deformation} of the holomorphic canonical connection   is
   the left invariant, torsion free, holomorphic  Koszul connection $\mathfrak D^\zeta$ on $T^{1,0}H$ 
 defined by
\begin{equation}\label{E:Cconn}
\mathfrak D^\zeta_uv
=\frac1{2}\left([u,v]+\zeta\left([u_{\mathfrak m^{1,0}},v_{\mathfrak l^{1,0}}]+[v_{\mathfrak m^{1,0}},u_{\mathfrak l^{1,0}}]\right)\right),
\end{equation}
where  $u=u_{\mathfrak l^{1,0}}+u_{\mathfrak m^{1,0}}$, $v=v_{\mathfrak l^{1,0}}+v_{\mathfrak m^{1,0}}$ are the decompositions in $\mathfrak l^{1,0}\oplus\mathfrak m^{1,0}=\mathfrak h^{1,0}$.
\end{definition}
\begin{prop}\label{P:Cheegercompl} 
\
\begin{enumerate}
    \item[a)] Let $G$ be a real Lie group and   $\mathfrak g=\mathfrak k\oplus \mathfrak p$  a reductive decomposition. Then  for any $\delta\in\bR$, the connection $\nabla^\delta$ is $\bR$-complete.
\item[b)] Let $H$ be a complex Lie group with a 
reductive decomposition  $\mathfrak l\oplus\mathfrak m=\mathfrak h$. Then  for any $\zeta\in\bC$ the holomorphic connection $\mathfrak D^\zeta$ is $\bC$-complete.
\end{enumerate}
\end{prop}
\begin{proof} a)
    Let $\varphi_\delta:\mathfrak g\times\mathfrak g\ra\mathfrak p$ be defined by
    $\varphi_\delta(x,y)=\frac{\delta}{2}\left([x_\mathfrak p,y_\mathfrak k]+[y_\mathfrak p,x_\mathfrak k]\right)$.
    
    Then $\varphi_\delta$ is symmetric, bilinear,  and  obviously $\left.\varphi_\delta\right|_{\mathfrak p\times\mathfrak p}\equiv0$. 
    It follows from Proposition~\ref{P:complete}  that the  vector field $\mathcal V_\delta$,
    corresponding to the quadratic polynomial map
    $\varphi_\delta(x,x)=\delta[x_\mathfrak p,x_\mathfrak k]$ is $\bR$-complete.
Then     Corollary~\ref{C:conncomplete}  yields the $\bR$-completeness of $\nabla^\delta$.

   b) The proof in the complex case is the same with obvious changes in notation, using 
    Proposition~\ref{P:complete} and Corollary~\ref{C:Ccomplete}.
\end{proof}
\begin{theorem}\label{T:reddef}
 Let $G$ be a real Lie group with a reductive decomposition $\mathfrak g=\mathfrak k+\mathfrak p$,  $x_0=x_{0\mathfrak k}+x_{0\mathfrak p}\in\mathfrak g$, $a\in G$, $v:=(L_a)_*(x_0)$, $\eta:G\ra G_\bC$ the universal complexification  and $\delta\in\bR$. Define 
the left invariant Koszul connection $\nabla^\delta$ by \eqref{E:conn} and the operator $\Lambda:\mathfrak g\ra\mathfrak g$ by $\Lambda x:=\delta x_\mathfrak k+x$.   Then
\begin{enumerate}
    \item[a)]  the unique $\nabla^\delta$-geodesic $\gamma_v$ with $\dot\gamma_v(0)=v$ has the form
 \begin{equation}\label{E:explgeod}
    \gamma_v(t)=a\exp_G(t\Lambda x_0)\exp_G(-t\delta x_{0\mathfrak k}),
\end{equation} 
\item[b)] $\eta\circ\gamma_v$ extends holomorphically to $\bC\ra G_\bC$, the  polar map $\psi_\eta$ is defined on $TG$ and is expressed by the formula
\begin{equation}\label{E:polar}
  \psi_\eta(v)=\eta(a)\exp_{G_\bC}(i\eta_*(\Lambda x_0))\exp_{G_\bC}(-i\delta \eta_*(x_{0\mathfrak k})).   
\end{equation}
 \item[c)] $P:=(\psi_\eta)_*^{-1}(T^{1,0}G_\bC)\subset \bC T(TG)$ defines  a $G$-invariant ac-polarization on $TG$, in particular $\nabla^\delta$ is $\mathcal P$-complete.
 \end{enumerate}
  \end{theorem}
\begin{proof} a) 
Since $\nabla^\delta$  is left invariant, we can assume that $a=e$ is the unit element. According to    Proposition~\ref{P:leftinvconn}
the curve
\begin{equation}\label{E:xcurve}
    x(t):=(L_{\gamma(t)})^{-1}_*(\dot\gamma(t))
\end{equation}
must satisfy $\dot x=-\nabla^\delta_xx$, i.e. the system
\begin{align}
\dot x_{\mathfrak k}&=0\\
\dot x_\mathfrak p&=-\delta[x_\mathfrak p,x_\mathfrak k]=\delta[x_\mathfrak k,x_\mathfrak p],
\end{align}
so with $A(t):=t ad_{\delta x_{0\mathfrak k}}:\mathfrak g\ra\mathfrak g$ and $\eta(t):=\exp(t\delta x_{0\mathfrak k})$ we get
\begin{equation}\label{E:++}
 x_\mathfrak k(t)=x_{0\mathfrak k}\quad\text{and}
 \qquad x_\mathfrak p(t)=e^{A(t)}x_{0\mathfrak p}=Ad_{\eta(t)}(x_{0\mathfrak p}), \quad x(t)=
 Ad_{\eta(t)}(x_0).
\end{equation}
Let $\chi(t):=\gamma(t)\eta(t)$. Then from \eqref{E:xcurve} and \eqref{E:++}
\begin{align*}
  \dot\chi(t)&=(R_{\eta(t)})_*\dot\gamma(t)+ 
  (L_{\gamma(t)})_*\dot\eta(t)\\
 &=(R_{\eta(t)})_*(L_{\gamma(t)})_*(x(t))+
 (L_{\gamma(t)})_*(L_{\eta(t)})_*(\delta x_{0\mathfrak k})\\
&=(R_{\eta(t)})_*(L_{\gamma(t)})_*Ad_{\eta(t)}x_0+
 (L_{\chi(t)})_*(\delta x_{0\mathfrak k})\\
 &=(R_{\eta(t)})_*(L_{\gamma(t)})_*(L_{\eta(t)})_*(R_{\eta(t)})^{-1}_*(x_0)+
 (L_{\chi(t)})_*(\delta x_{0\mathfrak k})
=(L_{\chi(t)})_*(\Lambda x_0).
\end{align*}
That means the curve $\chi(t)$ is a $1$-parameter subgroup $\chi(t)=\exp(t\Lambda x(0))$
yielding the claimed form \eqref{E:explgeod} of
$\gamma$.

b) Since $\eta$ is  a homomorhism we get
 \begin{equation}\label{E:etageod}
    \eta\circ\gamma_v(t)=\eta(a)\exp_{G_\bC}(t\eta_*(\Lambda x_0))\exp_{G_\bC}(-t\delta \eta_*(x_{0\mathfrak k})),
\end{equation}
But $\psi_\eta(v)=\eta\circ\gamma_v(i)$, so \eqref{E:etageod} proves part (b). 
Part c) follows from part b) and  Theorem~\ref{T:nabla-i-complete}.
\end{proof}
 \begin{remarks}
Such explicit formulas for geodesics and polar maps were already known in some specific situations i.e. for geodesics of deformations of the biinvariant Riemannian metric of a compact Lie group (\cite{DZ79} and  \cite{BLP24}) or deformations of the Killing form of a real semisimple Lie group \cite{HI06}.
The proof of Theorem~\ref{T:reddef} part a)  is a modified version of
Proposition 6.3 in \cite{BLP24} adjusted to our more general situation.   
     The proof of Proposition~\ref{P:Cheegercompl} also shows that
     the Euler-Arnold vector field of $\nabla^\delta$ is $\bC$-complete (hence $i\bR$-complete) in the complexified Lie algebra,   and then  Theorem~\ref{T:nabla-i-complete} would  yield the $\mathcal P$-completeness of $\nabla^\delta$.  

     The connections $\nabla^\delta$ provide genuinely new $\mathcal P$-complete connections, i.e. they are (mostly) not affine locally symmetric. Indeed the covariant derivative $(\nabla^\delta R^\delta)(x,y,z,w)$ of the curvature tensor $R^\delta(x,y)z$ of $\nabla^\delta$
     can be identified with 
      a  polynomial $P(\delta)=\sum^3_{j=0}a_j\delta^j$ in $\delta$ where $a_j$ are quadrilinear  maps $\mathfrak g\times\mathfrak g\times\mathfrak g\times\mathfrak g\ra \mathfrak g$ and
      $a_0=\nabla R^0$ is the covariant derivative of the curvature of the canonical connection.
    From Proposition~\ref{P:Cartanlocsymm} we know that for a $G$ that is not 2-step solvable, $a_0\not\equiv0$. Hence the kernel
    of the map \eqref{E:locsymornot}, denoted by $V$, is at most 3-dimensional.
    \begin{equation}\label{E:locsymornot}
    \bR^4\ni (\lambda_0,\lambda_1,\lambda_2,\lambda_3)\mapsto \sum^3_{j=0}\lambda_ja_j.
    \end{equation}
    Since for any pairwise different $\delta_1,\dots,\delta_4\in\bR$, the vectors
    $v_{\delta_j}:=(1,\delta_j,\delta^2_j,\delta^3_j)$ form a basis in $\bR^4$, $V$ can contain at most 3 out of such $v_\delta$ vectors. That means: if $G$ is not 2-step solvable,
    there are at most 3 different $\delta$, for which $\nabla^\delta$ is locally symmetric.
 \end{remarks}
 In the next section we shall say more on deformations of (pseudo)-Riemannian metrics.
\subsection{Deformations of biinvariant pseudo-Riemannian metrics}
\begin{definition}\label{D:Cheegerm}
Let $G$ be a real Lie group with Lie algebra $\mathfrak g$ and   biinvariant pseudo-Riemannian metric $\lr$ on $G$. Let $s\in\bR\setminus\{0\}$ and $\mathfrak k\subset\mathfrak g$  a Lie subalgebra so that $\left.\lr\right|_\mathfrak k$ is nondegenerate.  The 
{\it $(s,\mathfrak k)$-deformation} of $\lr$ 
is the left invariant pseudo-Riemannian metric $g_{s,\mathfrak k}$  defined by the formula
\begin{equation}\label{E:Rmetricdefor}
    g_{s,\mathfrak k}(u,v):=\langle u_{\mathfrak k^\perp},v_{\mathfrak k^\perp}\rangle+s\langle u_\mathfrak k, v_\mathfrak k\rangle=\langle u,v\rangle+(s-1)\langle u_\mathfrak k,v_\mathfrak k\rangle\qquad u,v\in\mathfrak g.
\end{equation}
\end{definition}
$g_{1,\mathfrak k}$ is then the original biinvariant  pseudo-Riemannian metric.  
\begin{remark}  
Such (and more general) deformations of Riemannian metrics were 
 first used by  Berger \cite{B61} and Cheeger \cite{C73} and since then it became an important tool in differential geometry to construct nonnegatively curved metrics. Here and in the previous section  such deformed metrics (or connections) are important, because 
their geodesics and the polar map can be  explicitly determined and that these deformed metrics (connections) are all $\mathcal P$-complete.
 \end{remark}
\begin{theorem}\label{T:Cheegerdeformmetriccompl}
 Let $G$ be a real Lie group with Lie algebra $\mathfrak g$ equipped with a biinvariant pseudo-Riemannian metric  $\lr$. Suppose  $\mathfrak k\subset\mathfrak g$ is a Lie subalgebra  so that the restriction of $\lr$ to $\mathfrak k$ is nondegenerate. 
Then 
\begin{enumerate}
    \item[a)] $\mathfrak g=\mathfrak k\oplus\mathfrak k^\perp$ is a reductive decomposition.
\item[b)]
Let $s\in\bR\setminus\{0\}$. Then the
 Levi-Civita connection of the  metric $g_{s,\kappa}$ defined   by  $\eqref{E:Rmetricdefor}$ is  the $\delta$-deformation of the canonical connection of $G$ in the sense of $\eqref{E:conn}$, corresponding to the reductive decomposition  $\mathfrak g=\mathfrak k\oplus\mathfrak k^\perp$ 
and $\delta=s-1$.
\item[c)]  $g_{s,\kappa}$  is $\mathcal P$-complete.
\end{enumerate}
\end{theorem}
\begin{proof} 
a)
Let $u,w\in\mathfrak k$, $v\in\mathfrak k^\perp$. Then $[u,w]\in\mathfrak k$, hence biinvariance of $\lr$ implies
\begin{equation}
    \langle[u,v],w\rangle=-\langle v,[u,w]\rangle=0
\end{equation}
Therefore
\begin{equation}
    [\mathfrak k,\mathfrak k^\perp]\subset\mathfrak k^\perp.
\end{equation}
This shows that $\mathfrak g=\mathfrak k\oplus\mathfrak k^\perp$ is a reductive decomposition.

b) From the definition of  $g_{s,\mathfrak k}$ we get
\begin{equation}
     g_{s,\mathfrak k}(u,v):=\langle u_{\mathfrak k^\perp},v_{\mathfrak k^\perp}\rangle+s\langle u_\mathfrak k, v_\mathfrak k\rangle=\langle u,v_{\mathfrak k^\perp}\rangle+s\langle u,v_\mathfrak k\rangle=\langle u,\Lambda v\rangle,
     \end{equation}
    where $\Lambda:\mathfrak g\ra\mathfrak g$, $\Lambda v:=sv_\mathfrak k+v_{\mathfrak k^\perp}=(s-1)v_\mathfrak k+v$. 
   Let  now $u,v,w\in\mathfrak k$. Then
   \begin{equation}
       \langle w,\Lambda ad^*_uv\rangle=g_{s,\mathfrak k}(w,ad^*_uv)=g_{s,\mathfrak k}(ad_uw,v)=\langle ad_uw,\Lambda v\rangle=-\langle w, ad_u\Lambda v\rangle
   \end{equation}
This yields 
    \begin{equation}\label{E:ad*}  
    ad^*_uv=-\Lambda^{-1}ad_uv.
 \end{equation}    
Now 
$$
ad_u\Lambda v=[u,v_{\mathfrak k^\perp}+sv_\mathfrak k]=[u,v]+(s-1)[u,v_\mathfrak k]=[u,v]+(s-1)[u_\mathfrak k,v_\mathfrak k]+(s-1)[u_{\mathfrak k^\perp},v_\mathfrak k].
$$
Hence
$$
ad_u\Lambda v+ad_v\Lambda u=(s-1)([u_{\mathfrak k^\perp}, v_\mathfrak k]+[v_{\mathfrak k^\perp}, u_\mathfrak k])\in\mathfrak k^\perp.
$$
Therefore
\begin{equation}\label{E:lambdainverse}
\Lambda^{-1}(ad_u\Lambda v+ad_v\Lambda u)=ad_u\Lambda v+ad_v\Lambda u=
(s-1)([u_{\mathfrak k^\perp}, v_\mathfrak k]+[v_{\mathfrak k^\perp}, u_\mathfrak k]).
\end{equation}
From Corollary~\ref{C:EAequation} and  formulas \eqref{E:ad*} and \eqref{E:lambdainverse}   we get
    $$
    \nabla_uv=\frac{1}{2}\left([u,v]-ad^*_uv-ad^*_vu\right)=
    \frac{1}{2}\left([u,v]+(s-1)([u_{\mathfrak k^\perp}, v_\mathfrak k]+[v_{\mathfrak k^\perp}, u_\mathfrak k])\right)
    $$
 proving part b).  

c) This follows from part b) and Theorem~\ref{T:reddef} c).
\end{proof}
The  proof of part b) is  a   modification of  parts of \cite{BLP24} taylored to our more general situation. 
\subsection{Semisimple Lie groups}\label{Ss:semisimple}
   Let $G$ be a connected, real, semisimple Lie group with Lie algebra $\mathfrak g$ and denote by  $\langle.,.\rangle$   the Killing form  of $G$.
   When $G$ is noncompact, the Cartan decomposition
   $\mathfrak g=\mathfrak k\oplus\mathfrak p$   with respect to a maximal compact Lie subalgebra $\mathfrak k$ provides a reductive decomposition of $\mathfrak g$.
When $G$ is compact let $G'$ be a noncompact real form in the universal complexification $G_\bC$ of $G$ and let $K=G\cap G'$. The Cartan decomposition $\mathfrak g'=\mathfrak k\oplus\mathfrak p'$
   with respect to $K$ yields the   reductive decomposition $\mathfrak g=\mathfrak k\oplus\mathfrak p$, where $\mathfrak p:=i\mathfrak p'$.
These  special examples of Definition~\ref{D:Cheegerm}
    were studied by  Halverscheid and Iannuzzi in \cite{HI06} and \cite{HI09}.
In their notation 
    our deformed pseudo-Riemannian metric  $g_{s,\mathfrak k}$ was denoted by $\nu_m$ with 
    $m=-s$).  Among other things they determined the explicit form \eqref{E:explgeod} of the geodesics, the Levi-Civita connection and    the polar map \eqref{E:polar} as well. 
    In the noncompact case they also showed that for all $s\not=0$, the metric $g_{s,\mathfrak k}$ is not entire but there exists a maximal domain  $\Omega_s$, 
    (with  some very complicated complex geometry)
     where the ac-structure can be defined. 
In the compact case they calculated the sectional curvatures of the metric $-g_{s,\mathfrak k}$ (which is Riemannian when $0<s$) and concluded that for $4/3<s$ some sectional curvatures are negative. This 
(in light of Theorem 4.3 in \cite{LSz91}) 
implies that for such $s$, the Riemannian metric  $-g_{s,\mathfrak k}$ cannot be entire. 
 Since rescaling a (pseudo)-Riemannian metric with a scalar doesn't change the Levi-Civita connection, we get that $g_{s,\mathfrak k}$
is not entire either ($4/3<s$).
 On the other hand for any connected (compact or noncompact) semisimple Lie group $G$,
    Theorem~\ref{T:Cheegerdeformmetriccompl} implies that  the metrics
      $g_{s,\mathfrak k}$ are  $\mathcal P$-complete for all $0\not=s\in\bR$.   
\section{Vector fields}\label{S:vfields}
\subsection{Complete, homogeneous, quadratic vector fields}\label{Ss:completevf}
\begin{prop}\label{P:complete} Denote by $\mathbb K$ either the field of real or complex numbers. Let $n$ be a positive integer.
Suppose   $\mathbb K^n=W\oplus V$ is a decomposition  to $\mathbb K$-vector subspaces.
Let $A:\mathbb K^n\times \mathbb K^n\ra V$ be a  symmetric, $\mathbb K$-bilinear map. 
Define  $\varphi(z):=A(z,z)$, $z\in\mathbb K^n$ and
 denote by $\mathcal V(z)$  the  corresponding vector field on $\mathbb K^n$.
 Suppose 
 $\left.A\right|_{V\times V}\equiv 0$.
Then $\mathcal V$
is $\mathbb K$-complete i.e.  all integral curves of
\begin{equation}\label{E:phi}
    \dot z=\mathcal V(z)
\end{equation}
are defined on $\mathbb K$.
\end{prop}
\begin{proof} For  $z\in \mathbb K^n$ write  $z=z_w\oplus z_v$ with $z_w\in W$  and $z_v\in V$. Then $\varphi(z)=A(z,z)=A(z_w,z_w)+2A(z_w,z_v)$.
  Let $z(\zeta)$ be a solution of \eqref{E:phi}. Then  equation \eqref{E:phi} can be rewritten as 
  \begin{align}
 \dot z_w&=0,\label{E:v1}\\
 \dot z_v&=\varphi(z_w)+2A(z_w,z_v).\label{E:w}
  \end{align}
  $\eqref{E:v1}$ implies that $z_w\equiv c$ constant. But then \eqref{E:w} reads as $\dot z_v=\varphi(c)+2A(c,z_v)$  a system of first order inhomogeneous linear ODE's  with constant coefficients, so its solutions are indeed defined on $\mathbb K$.
\end{proof}
Denote by $\mathcal CHQ_n$ the set of $\mathbb C$-complete complex homogeneous quadratic vector fields on $\bC^n$. Let  $\mathcal C_n$ be the set of those complex homogeneous quadratic vector fields on $\bC^n$ that corresponds to a decomposition and quadratic polynomial map  $\varphi$ as in Proposition~\ref{P:complete}, where $\dim V$ can be $1,2,\dots, n$.   Clearly $\mathcal CHQ_n$  and $\mathcal C_n$ are both  $GL(n,\bC)$-invariant and by Proposition \ref{P:complete}, $\mathcal C_n\subset\mathcal CHQ_n$. 
\begin{conj} $\mathcal C_n=\mathcal CHQ_n$.
\end{conj}
At  least for $n=2$, we can show  this to be true.
\begin{theorem}\label{T:complete2} Let $P_1(z_1,z_2)$, $P_2(z_1,z_2)$ be complex quadratic, homogeneous polynomials. 
 The  vector field $\mathcal V(z_1,z_2)=P_1\partial_{z_1}+P_2\partial_{z_2}$ is $\bC$-complete iff after an appropriate linear change of coordinates $\mathcal V$ has the form   
 $$\mathcal V_1(z_1,z_2)=z_1z_2\partial_{z_2}
  \quad \text{or}\quad 
 \mathcal V_2(z_1,z_2)=z^2_1\partial_{z_2}.
 $$
 \end{theorem}
 Abate and Tovena  \cite{AT11} gave a complete list (although not a detailed proof) of all quadratic homogeneous vector fields on $\bC^2$ up to the $GL(2,\bC)$ action. One can check that   only two vector fields on their list are $\bC$-complete. These are the fields (using the notations of 
 \cite{AT11}) $1_{00}$ which is our $\mathcal V_2$ and $2_{001}$ which is our $\mathcal V_1$. From this, our theorem follows. But we give a direct proof here not using the above classification.  See more on the paper \cite{AT11}) in Subsection~\ref{Ss:iRcomplvf}.
 \begin{proof}
 If $\mathcal V$ is one of this form, the first coordinate of the trajectories must be constant and the second coordinate has the form $ae^{bz}$ in the first case and  $az+b$ in the second case. Hence $\mathcal V$ is $\bC$-complete.
  
  To prove the opposite direction, we use the fact:  trajectories of $W(z)=cz^2\partial_z$, $z\in\bC$  are 
  never entire (only meromorphic), except if $c=0$. 
  Hence  a $\bC$-complete  vector field must vanish along  invariant complex lines (i.e. 1-dimensional complex subspaces to which $\mathcal V$ is tangential). 
  
    Since $P_j$ are quadratic homogenenous  polynomials, we can write them in the form $P_j=\ell_jL_j$, where $\ell_j, L_j\in (\bC^2)^*$. The singularity set  of $\mathcal V$ is  $S:=\{z\in\bC^2\mid \mathcal V(z)=0\}$. Thus
    $S=(\ker \ell_1\cup\ker L_1)\cap (\ker \ell_2\cup\ker L_2)$.

    a) Suppose 
    $S=\{(0,0)\}$. We show that   $\mathcal V$ is incomplete.  Since $S$ is just a point, the linear functionals $\ell_1,\ell_2$ and also $L_1$, $L_2$ must be linearly independent. 

Let $\Phi(z):=(\ell_1(z),\ell_2(z))$. Then $P_j(\Phi^{-1}(w))=w_jL_j(\Phi^{-1}(w))=w_j(a_jw_1+b_jw_2)$ with appropriate
$a_j, b_j\in\bC$ and $U:=\Phi_*\mathcal V=w_1L_1(\Phi^{-1}w)\partial_{w_1}+w_2L_2(\Phi^{-1}w)\partial_{w_2}$. Now
$p_0:=(b_2-b_1,a_1-a_2)\not=(0,0)$ since $L_1$ and $L_2$ are linearly independent. 
Furthermore a short calculation shows that
$W(\zeta p_0)=\zeta^2(\det\Phi)p_0$. Hence $W$ is a tangential vector field along the complex line $\zeta p_0$, and it is not identically zero. Therefore $W$ and consequently $\mathcal V$ is incomplete.

b) Suppose $\dim S=1$ and $P_j=\ell L_j$, with $\ell, P_j\in(\bC^2)^*$ and $L_1, L_2$ are linearly independent. Again we are going to prove that $\mathcal V$ is incomplete. First we show that after some linear change of coordinates, we can assume  $\ell(z)=z_1$ and hence
\begin{equation}\label{E:V2}
\mathcal V(z)=z_1L_1(z)\partial_{z_1}+z_1L_2(z)\partial_{z_2}.    
\end{equation}
If $\ell(z)=\alpha z_1$, we can rewrite $P_j$ as  
$P_j=z_1(\alpha L_j)$. Suppose $\ell(z)=\alpha z_1+\beta z_2$ with $\beta\not=0$. Similarly as in case a), let
$\Phi(z)=(z_1,\ell(z))$. Then $U(w)=\Phi_*\mathcal V(w)=P_1(\Phi^{-1}w)\partial_{w_1}+(\alpha P_1(\phi^{-1}w)+\beta P_2(\Phi^{-1}w))\partial_{w_2}=w_1\tilde L_1(w)\partial_{w_1}+w_1(\alpha\tilde L_1(w)+\beta \tilde L_2(w))\partial_{w_2},$
where $\tilde L_j(w)=L_j(\Phi^{-1}w)$.
As one readily sees, $\tilde L_1$ and $\alpha\tilde L_1+\beta \tilde L_2$ are linearly independent, since $L_1$ and $L_2$ are. Thus $U$ has the form we claimed. 

Back to our original notation, suppose that $\mathcal V$ is of the form \eqref{E:V2} and $L_1(z)=a_1z_1+b_1z_2$. If $b_1=0$, the equation for the first coordinate of a trajectory is $z'_1=a_1z^2_1$. This has only meromorphic solutions ($a_1$ cannot be zero, otherwise we would have $L_1=0$ and $L_1$, $L_2$ would be linearly dependent), so $\mathcal V$ is incomplete. Suppose $b_1\not=0$. To find  invariant lines, we need to solve
$$
z_1L_1(z)-z_0z_1=0,\quad z_1L_2(z)-z_0z_2=0.
$$
Substituting $(1,t)$ leads to the equation
\begin{equation}\label{E:t}
L_2(1,t)-L_1(1,t)t=0.    
\end{equation}
Since $b_1\not=0$, this is quadratic in $t$, hence it has a solution $t_0$. 
Then $\mathcal V(\zeta,t_0\zeta)=\zeta^2(L_1(1,t_0)\partial_{z_1}+L_2(1,t_0)\partial_{z_2})=\zeta^2L_1(1,t_0)(\partial_{z_1}+t_0\partial_{z_2})$, hence 
$\mathcal V$ will be tangential to the complex line
$B=(\zeta,t_0\zeta)$. If $L_1(1,t_0)$ were zero,
equation~\eqref{E:t} would yield  
$(1,t_0)\in\ker L_1\cap\ker L_2$, and so $L_1$, $L_2$ were linearly dependent, a contradiction. Therefore $L_1(1,t_0)\not=0$. Thus $\left.\mathcal V\right|_B\not\equiv0$ implying that indeed $\mathcal V$ is incomplete.

c) Suppose $\dim S=1$ and $P_j=c_j\ell L$, with $c_j\in\bC$, $\ell, L\in(\bC^2)^*$ and $\ell, L$ are linearly independent. We are going to show that $\mathcal V$ is complete iff 
$\ell(c_1,c_2)=0$ and $L(c_1,c_2)\not=0$ or reversed.
In the special case $\ell(z)=z_1$, $L(z)=z_2$ we have $\mathcal V(z_1,z_2)=c_1z_1z_2\partial_{z_1}+c_2z_1z_2\partial_{z_2}$. If $c_1$ or $c_2$ is zero, we know $\mathcal V$ is complete. If both are different from zero, with $p_0=(c_1,c_2)$
we have $\mathcal V(\zeta p_0)=\zeta^2c_1c_2p_0$. This shows that $\mathcal V$ is tangential to the line spanned by $p_0$ and it is not the zero vector field. Hence $\mathcal V$ is incomplete.
The general case is reduced to this case  again changing coordinates. 

Let $\ell(z)=\alpha z_1+\beta z_2$, $L(z)=\gamma z_1+\delta z_2$,
 $\Phi(z)=(\ell(z),L(z))$ and $U=\Phi_*\mathcal V$.
 Then 
 \begin{multline}    
 U(w)=(\alpha c_1P_1(\Phi^{-1}w)+\beta c_2P_2(\Phi^{-1}w)\partial_{w_1}+(\gamma c_1P_1(\Phi^{-1}w)+\delta c_2P_2(\Phi^{-1}w)\partial_{w_2})=\\
 =(\alpha c_1+\beta c_2)w_1w_2\partial_{w_1}+
 (\gamma c_1+\delta c_2)w_1w_2\partial_{w_2}=\ell(c_1,c_2)w_1w_2\partial_{w_1}+L(c_1,c_2)w_1w_2\partial_{w_2}. 
 \end{multline}
What was said above in the special case shows: $U$ (consequently $\mathcal V$ as well) is complete iff precisely one of $\ell(c_1,c_2)$ and $L(c_1,c_2)$ is zero.

d) Suppose $\dim S=1$ and $P_j=c_j\ell^2$, with $c_j\in\bC$, $\ell\in(\bC^2)^*$. 
The special case $\ell(z)=z_1$, $\mathcal V(z_1,z_2)=c_1z^2_1\partial_{z_1}+
c_2z^2_1\partial_{z_2}$ is complete iff $c_1=0$ ($z'_1=c_1z^2_1$). 

The general case again can be reduced to this. Let $\ell(z)=\alpha z_1+\beta z_2$. If $\alpha=0$, $\mathcal V$ will be complete iff $c_2=0$ as well ($z'_2=c_2\beta^2z^2_2$). So we can assume $\alpha\not=0$. Let $\Phi(z)=(\ell(z),z_2)$.  Then $\ell(\Phi^{-1}w)=w_1$ and
$U=\Phi_*\mathcal V(w)=(\alpha P_1(\Phi^{-1}w)+
\beta P_2(\Phi^{-1}w))\partial_{w_1}+P_2(\Phi^{-1}w)\partial_{w_2}=(\alpha c_1 w^2_1+\beta c_2w^2_1)\partial_{w_1}+c_2w^2_1\partial_{w_2}=\ell(c_1,c_2)w^2_1+c_2w^2\partial_{w_2}$.
The special case implies that $U$ (and thus $\mathcal V$) is complete iff $\ell(c_1,c_2)=0$.
 \end{proof}
The homogeneous quadratic vector fields on $\mathbb C^2$ forms  a complex vector space $\mathcal{HQ}_2$ of dimension 6 on which $GL(2,\mathbb C)$ acts. The set of complete vector fields $\mathcal CHQ_2$ is a  very small subset in $\mathcal{HQ}_2$. Easy calculations show that the stabilizer of $\mathcal V_1(z)=z_1z_2\partial_{z_2}$ and 
$\mathcal V_2(z)=z^2_1\partial_{z_2}$ are 
$$
\text{Stab} \mathcal V_1=\left\{\left(\begin{matrix}
 1&0\\0&\delta    
\end{matrix}\right)\mid \delta\in \mathbb C_*\right\},
\qquad
\text{Stab} \mathcal V_2=\left\{\left(\begin{matrix}
 \alpha&0\\\gamma&\alpha^2    
\end{matrix}\right)\mid \alpha\in \mathbb C_*,\gamma\in\mathbb C\right\}.
$$
Hence the dimension of the orbit
of $\mathcal V_1$ (resp. $\mathcal V_2)$ is $3$ (resp. $2$).
  \subsection{\texorpdfstring{$i\bR$}P-complete vector fields}\label{Ss:iRcomplvf}
 \begin{definition}\label{D:iRcomplete} Let
  $\mathcal W$ be a smooth, real vector field on $\bR^n$. A trajectory of $\mathcal W$ is $\bC$-complete  if it extends  to
a holomorphic map $\bC\ra \bC^n$.
We say that  the $\mathcal W$-trajectory $c$, $c(0)=p$ is  {\it i$\bR$-complete} at $p=c(0)$ or that $\mathcal W$ is {\it i$\bR$-complete} at $p$
 if   $c$,  extends holomorphically  to some open neighborhood of $i\bR$. 
  $\mathcal W$ is called {\it $\bC$-complete}
  (resp. {\it i$\bR$-complete}), if every trajectory is $\bC$-complete (resp. {\it i$\bR$-complete}).
\end{definition}
When $\mathcal W$ is both $\bR$-complete and $i\bR$-complete, then necessarily $\mathcal W$ will be $\bC$-complete.
\begin{prop}\label{P:t^2+a}  
Let  $A,B\in\bR$. Consider the following vector fields on $\bR$.

(a)
$\mathcal{W}(y):=(Ay^2+B)\partial_y$  is  $i\bR$-complete  iff   $A=0$ or $\frac{B}{A}\ge0$. When  $\frac{B}{A}<0$, $\mathcal W$ is  $i\bR$-complete at every point except at $p=0$.

(b)  
$\mathcal{W}(y):=(Ay^2+By)\partial_y$ is $i\bR$-complete  iff $A=0$.
When $A\not=0$, $\mathcal W$ is 
 $i\bR$-complete at every pont  except  at 
 $p=-\frac{B}{2A}\in \bR$.
\end{prop}
\begin{proof}
(a) When  $A=0$, $\mathcal W$ will be trivially $\bC$-complete.
Suppose  $A\neq 0$ and $y$ is a trajectory of $\mathcal{W}$. Let $x(t)=y(t/A)$. Then $x(t)$ satisfies the equation
\begin{equation}\label{E:x^2+a}
    \dot{x}(t)=x(t)^2+a,
\end{equation}
where $a=\frac{B}{A}$.  Hence it suffices to deal with the vector field $\mathcal V(x)=(x^2+a)\partial_x$.  

If $a=0$  all trajectories of $\mathcal V$ have the form
\begin{equation}
    x(t)=\frac{x(0)}{1-x(0)t}.
\end{equation}
When $x(0)=0$, $x(t)$ is constant, hence $\bC$-complete. If $x(0)\not=0$, $x$ extends meromorphically to $\bC$ with a sole singularity at $1/x(0)\in\bR$, so  $\mathcal V$ is $i\bR$-complete.

If $a>0$, the trajectories are 
\begin{equation}
x(t)=\sqrt{a}\tan\left(\sqrt{a}t+\arctan\left(\frac{x(0)}{\sqrt{a}}\right)\right).
\end{equation}
The poles of $\tan(z)$ are located on the real line at $\pi/2+k\pi$, where $k\in\bZ$, and the range of $\arctan(t)$ is $(-\pi/2,\pi/2)$. So $x(t)$ is holomorphic in a neighborhood of $i\bR$, hence $\mathcal V$ is again $i\bR$-complete.

If $a<0$ let $b=\sqrt{-a}$. Thus $x(t)$ is a solution of $\dot{x}=x^2-b^2$. If $x(0)=\pm b$, then $x(t)\equiv \pm b$, so $x(t)$ is $\bC$-complete. Otherwise 
$$x(t)=b\frac{(x(0)+b)e^{-bt}+(x(0)-b)e^{bt}}{(x(0)+b)e^{-bt}-(x(0)-b)e^{bt}}.$$
So $x(t)$ has a meromorphic extension to $\bC$ having a singularity at z iff
\begin{equation}\label{E:pole}
    \frac{x(0)+b}{x(0)-b}=e^{2bz}.
\end{equation}
If $x(0)\neq0$, equation \eqref{E:pole} has no solution in a small enough open neighborhood of the imaginary axes yielding the $i\bR$-completenes of $\mathcal V$ at $0\not=x(0)$. On the other hand if $x(0)=0$, the trajectory  has the form
\begin{equation}
    x(t)=-b\tanh(bt),
\end{equation}
which does admit singular points on the imaginary axis, so $\mathcal V$  is not $i\bR$-complete at $0$.

The proof of (b) is similar and left to the reader.
\end{proof}
\begin{prop}\label{P:q-v}
    Let $v\in\bR^n$, and $q:\bR^n\times\bR^n\rightarrow\bR$  a symmetric bilinear form. Then the vector field $\mathcal{W}(x)=\sum_j(q(v,x)x_j-q(x,x)v_j)\partial_{x_j}$ is not $i\bR$-complete at $u\in\bR^n$ iff   $q(v,u)=0$  and $q(v,v)q(u,u)< 0$. In particular, if $v$ is lightlike, the vector field is $i\bR$-complete.
\end{prop}
\begin{proof}
     $q(x,x)$ will be a first integral of $\mathcal{V}$, since 
    for an arbitrary trajectory $x$
    $$q(x,x)'=2q(\dot{x},x)=q(q(v,x)x-q(x,x)v,x)=q(q(v,x)x,x)-q(q(x,x)v,x)=0.$$    
    If $q(v,.)\equiv 0$, then $\mathcal{V}(x)=-q(x,x)v=-q(x(0),x(0))v$   and the trajectories take the form $x(t)=x(0)-q(x(0),x(0))tv$. That shows $\mathcal W$ is $\bC$-complete.

    If $q(v,.)\not\equiv 0$, there is a $w\in\bR^n$ such that $q(v,w)=1$.  Denote $\ker(q(v,.))$ by $V$. Then
    $\bR^n=V\oplus \langle w\rangle$. Now let $x\in\bR^n$. Then  $x=x_V+x_ww$, where $x_V\in V$. Hence $x_w=q(x,v)$ and  for $v=x$ we get $v=v_V+v_ww=v_V+q(v,v)w$. Now let $x(t)$ be a trajectory. Since $q(x,x)$ is constant,   $x$ must satisfy the following system of ODE. 
    \begin{equation}\label{E:q-v-w}
        \dot{x}_w=q(v,x)x_w-v_wq(x,x)=x_w^2-q(v,v)q(x(0),x(0)).
    \end{equation}
\begin{equation}\label{E:q-v-V}
        \dot{x}_V=x_wx_V-q(x(0),x(0))v_V.
    \end{equation}
Suppose $q(v,v)=0$.  Then \eqref{E:q-v-w} implies that $x_w$  takes the form $x_w(t)=\frac{x_w(0)}{1-x_w(0)t}$. So,
    if $x_w(0)=0$, then $x_w\equiv0$ and equation \eqref{E:q-v-V} shows the trajectory $x$ is $\bC$-complete. If $x_w(0)\not=0$, we still get that $x_w(t)$ extends holomorphically to the strip $U_{\eta}=\{\zeta\in\bC:|\Re \zeta|<\eta\}$  where $ \eta=1/|x_w(0)|\}$. Equation \eqref{E:q-v-V} then defines an inhomogeneous, linear system of ODEs with holomorphic coefficients in $U_{\eta}$. Therefore its solutions are also holomorphic on $U_{\eta}$. Hence  $\mathcal W$ is $i\bR$-complete.

    Suppose $q(v,v)\neq 0$. 
    Then \eqref{E:q-v-w}, Proposition \ref{P:t^2+a} (a) and \eqref{E:q-v-V} shows that the trajectory $x(t)$ does not extend holomorphically to a neighborhood of $i\bR$  iff $x_w(0)=0$ and $q(v,v)q(x_V(0),x_V(0))<0$, proving the Proposition.
\end{proof}
\begin{prop}\label{P:l-L-Omega} 
Let $n$ be a positive integer.
\begin{enumerate}
    \item[a)] For  an $l\in {(\bR^n)}^*$
 define the  map $\varphi:\bR^n\ra\bR^n$ by
$\varphi(x):=l(x)x$ and let $\mathcal W$ be the corresponding
   vector field on $\bR^n$.
Then $\mathcal W$ is $i\bR$-complete.
More generally we have.
\item[b)] Suppose   $\bR^n=W\oplus V$ is a decomposition  to vector subspaces. Let $A:W\times W\ra W$ and  $B:\bR^n\times\bR^n\ra V$ be  symmetric bilinear maps and $\left.B\right|_{V\times V}\equiv0$. Define $\psi(u):=A(u,u), u\in W$
and $\varphi(x):=A(x_w,x_w)\oplus B(x,x)$.
Denote by  $\mathcal V_w$ the corresponding vector field to $\psi$
 on $W$  
 (resp. of $\varphi$  on $\bR^n$ by $\mathcal W$). Assume that $\mathcal V_w$ is $i\bR$-complete.
 Then $\mathcal V$ is $i\bR$-complete. When $A=0$, 
$\mathcal V$ will be even  $\bC$-complete. 
\end{enumerate}
\end{prop}
\begin{proof} 
a)  If $l\equiv0$, $\mathcal W$ is trivially $\bC$ complete. Otherwise, let us write  $\bR^n=\langle u\rangle\oplus \ker(l)$, where $u\in \bR^n$ is such that $l(u)=1$. Every $x\in \bR^n$ can be writen  as $x=x_u u\oplus x_l $, where $l(x_l)=0$. Now  a curve $x(t)$, defined in a neighborhood of $0$,  will be an integral curve of $\mathcal W$ iff 
\begin{align}
 \dot x_u&=x^2_u,\label{E:ucomp}\\
 \dot x_l&=x_u x_l.\label{E:lcomp}
\end{align}
Proposition \ref{P:t^2+a}(a) shows that $x_u$ extends holomorphically to a neighborhood of the imaginary axes. Therefore \eqref{E:lcomp} is a
homogeneous linear system with holomorphic coefficients in this neighborhood, yielding that $\mathcal W$ is $i\bR$-complete at $x(0)$.

b) Let $x(t)$ be a trajectory of $\mathcal W$, defined in a neighborhood of $0$.
Then $x(t)$ must satisfy the following system of ODE's.
\begin{align}
   \dot x_w&=\psi(x_w)\label{E:first1} \\
   \dot x_v&=B(x,x)\label{E:2nd}
\end{align}
From our assumptions we know that $x_w$ extends holomorphically to some neighborhood of $i\bR$
and $B(x,x)=B(x_w,x_w)+2B(x_w,x_v)$.
Therefore the system \eqref{E:2nd} is an inhomogenous system of linear equations with holomorphic coefficients in some neighborhood of $i\bR$, yielding that $\mathcal W$ is $i\bR$-complete at $x(0)$.

If $A\equiv0$, $x_w$ must be constant, therefore the system \eqref{E:2nd}
has constant coefficients and so the solution exists on $\bC$.    
\end{proof}
The complex quadratic homogeneous vector fields on $\bC^2$ w.r.t. the $GL(2,\bC)$ action were classified  in the paper \cite{AT11}, Sect.9.
The classification is according to the number of characteristic directions (invariant complex lines through the origin) and the fact whether restriction to an invariant line is the zero vector field or not. There are 11 equivalent classes. In certain cases a type in this classification has one or two free parameters. If we allow these parameters to be real only, all of the vector fields in this list are complexifications of real quadratic vector fields. Among them two are $\bC$-complete, see Subsection~\ref{Ss:completevf} and Theorem~\ref{T:complete2}. Out of the remaining 9
classes we were able to determine in 6 cases whether the corresponding real homogeneous vector field is $i\bR$-complete or not.
 The tag indicates the classification type of the complexified vector field used in  \cite{AT11}. We did not deal with the more general question: what is the complete list of  the real quadratic homogeneous vector fields w.r.t. the $GL(2,\bR)$ action. 
\begin{prop}\label{P:ir-complete-dim2}
 Among the  quadratic, real homogeneous vector fields 
 \begin{equation}\tag{$\infty$}
 \mathcal W_1(x_1,x_2)=x^2_1\partial_{x_1}+x_1x_2\partial_{x_2}
\end{equation}
\begin{equation}\tag{$1_{10}$}
 \mathcal W_2(x_1,x_2)=-x^2_1\partial_{x_1}
 -(x^2_1+x_1x_2)\partial_{x_2} 
\end{equation}
\begin{equation}\tag{$2_{10\rho}$}
 \mathcal W_3(x_1,x_2)=-\rho x^2_1\partial_{x_1}+(1-\rho)x_1x_2\partial_{x_2}, \quad \rho\in\bR,\ \rho\not=0
\end{equation}
\begin{equation}\tag{$3_{\frac1{2}10}$}
\mathcal W_4(x_1,x_2)=\frac1{2}(-x^2_1+x_1x_2)\partial_{x_1}
 +\frac1{2}(-x^2_2+x_1x_2)\partial_{x_2} 
\end{equation}
\begin{equation}\tag{$3_{100}$}
 \mathcal W_5(x_1,x_2)=(x^2_1-x_1x_2)\partial_{x_1}
\end{equation}
\begin{equation}\label{E:first}\tag{$3_{\frac1{2}10}$}
\mathcal W_6(x_1,x_2)=-x^2_2\partial_{x_1}
 +x_1x_2\partial_{x_2} 
\end{equation}
$\mathcal W_1$, $\mathcal W_2$, $\mathcal W_3$, $\mathcal W_4$ are not $\bC$ complete but $i\bR$-complete on $\bR^2$;
$\mathcal W_5$ and $\mathcal W_6$ are not even $i\bR$-complete.
\end{prop}
\begin{proof} Let  $e_j$ be the standard bases in $\bR^2$  and
     $W=\langle e_1\rangle$ and $V=\langle e_2\rangle$.
     
 Proposition~\ref{P:l-L-Omega}(b) with 
 $$
 A(xe_1,ye_1):=xye_1,\qquad B((x_1,x_2),(y_1,y_2)):=\left(\frac{x_1y_2+y_1x_2}{2}\right)e_2
 $$ shows that $\mathcal W_1$
 is $i\bR$-complete. 
 
Proposition~\ref{P:l-L-Omega}(b) with 
 $$
 A(xe_1,ye_1):=-xye_1,\qquad B((x_1,x_2),(y_1,y_2)):=\left(-x_1y_1-\frac{x_1y_2+y_1x_2}{2}\right)e_2
 $$ shows that $\mathcal W_2$
 is $i\bR$-complete. 

Proposition~\ref{P:l-L-Omega}(b) with 
$$
A(xe_1,ye_1):=-\rho x_1y_1e_1,\qquad B((x_1,x_2),(y_1,y_2))=\frac{1-\rho}{2}(x_1y_2+y_1x_2)e_2
$$ shows that $\mathcal W_3$ is  $i\bR$-complete.
If $x=x_1e_1+x_2e_2$ is a trajectory of $\mathcal W_j$, $j=1,2,3$, then $\dot x_1=cx^2_1$
must hold with $c\not=0$, hence $x_1$ is not holomorphic on $\bC$, so these fields are not $\bC$-complete.

Proposition  \ref{P:q-v}, with $n=2$,  $q(x,y)=-\frac1{2}(x_1y_2+y_1x_2)$, $v=(1,1)$, $u=(1,-1)$ $q(v,v)=-1$, $q(u,u)=1$ shows that $\mathcal W_4$
is $i\bR$-complete, since $q(v,v)q(u,u)=-1<0$.
On the other hand
 $W=\langle e_1\rangle$ is a characteristic direction of  $\mathcal W_4$ but the restriction is not the identically zero  vector field, so $\mathcal W_4$ is not $\bC$-complete. 
Let $x(t)=x_1(t)e_1+x_2(t)e_2$ be a trajectory of 
$\mathcal W_5$. Then $\dot x_2\equiv0$, hence $x_2=x_2(0)$ and
  $x_1$ must satisfy $\dot x_1=x_1^2-x_2(0)x_1$. By Proposition~\ref{P:t^2+a}(b) this vector field is not $i\bR$-complete on $\bR$. 
 
 Proposition  \ref{P:q-v}, with $n=2$, $v=e_1$, $u=e_2$, 
$q(x,y)=x_1y_1+x_2y_2$ implies that $\mathcal W_6$  is not $i\bR$-complete, since $q(v,v)q(u,u)=1>0$.
\end{proof}
From Proposition~\ref{P:ir-complete-dim2} and Corollary~\ref{c:ircomplete} we obtain:
\begin{cor}
 Let $G$ be a two dimensional real Lie group (i.e. $\bR^2$ or $\mathcal A^+(1,\bR)$ the orientation-preserving affine parametrizations of the real line) and $\beta_j:\mathfrak g\times\mathfrak g\ra \mathfrak g$ the symmetric bilinear map that is the polarization of the quadratic, homogeneous map identified with the vector field $\mathcal W_j$, $j=1,2,3,4$ in Proposition~\ref{P:ir-complete-dim2}.
 Then 
 \begin{equation}
 \nabla_xy:=\frac1{2}[x,y]-\beta_j(x,y) \qquad x,y\in\mathfrak g   
 \end{equation}
 defines a $\mathcal P$-complete left invariant Koszul connection on $TG$.
\end{cor}  
\textbf{Declarations}

\textbf{Conflict of interest} On behalf of all authors the corresponding author states that there is no conflict of interest.

\textbf{Ethical approval} Not applicable.

\end{document}